\documentclass{article}

\usepackage{nomencl}
\usepackage{chngcntr}
\usepackage{amsmath}
\usepackage{amsthm}
\usepackage{amscd}
\usepackage{bbm}
\usepackage{enumerate}
\usepackage{amssymb}
\usepackage{palatino}
\usepackage{amscd}
\usepackage{verbatim}
\usepackage{mathrsfs}
\usepackage{esint}
\usepackage{xcolor}
\usepackage{hyperref}

\usepackage[T2A,T1]{fontenc}
\usepackage[utf8]{inputenc}

\DeclareSymbolFont{cyrillic}{T2A}{cmr}{m}{n}
\def\makecyrsymbol#1#2{%
  \begingroup\edef\temp{\endgroup
    \noexpand\DeclareMathSymbol{\noexpand#1}
    {\noexpand\mathalpha}{cyrillic}%
    {\expandafter\expandafter\expandafter
     \calccyr\expandafter\meaning\csname T2A\string#2\endcsname\end}}%
  \temp}
\expandafter\def\expandafter\calccyr\string\char#1\end{#1}

\makecyrsymbol\hardsign\cyrhrdsn
\makecyrsymbol\bighardsign\CYRHRDSN
\makecyrsymbol\iy\cyrery
\makecyrsymbol\bigiy\CYRERY
\makecyrsymbol\softsign\cyrsftsn
\makecyrsymbol\bigsoftsign\CYRSFTSN
\makecyrsymbol\zh\cyrzh
\makecyrsymbol\bigzh\CYRZH

\newtheoremstyle{dotless}{}{}{\itshape}{}{\bfseries}{}{ }{}

\newtheorem{thm}{Theorem}
\newtheorem{lem}[thm]{Lemma}

\newtheorem{cor}[thm]{Corollary}
\newtheorem{defn}[thm]{Definition}
\newtheorem{prop}[thm]{Proposition}

\theoremstyle{dotless}

\counterwithin{equation}{section}
\counterwithin{thm}{section}

\newcommand{\End}{\mathop{\text{End}}}
\newcommand{\Hom}{\mathop{\text{Hom}}}

\newcommand{\Z}{\mathbb{Z}}
\newcommand{\R}{\mathbb{R}}
\newcommand{\Q}{\mathbb{Q}}
\renewcommand{\C}{\mathbb{C}}
\newcommand{\Nm}{\mathrm{Nm}}

\newcommand{\F}{\mathbb{F}}

\newcommand{\id}{\mathop{\mathrm{id}}}

\newcommand{\N}{\mathbb{N}}

\newcommand{\Gal}{\mathop{\mathrm{Gal}}}

\newcommand{\pfrak}{\mathfrak{p}}
\newcommand{\qfrak}{\mathfrak{q}}
\newcommand{\mfrak}{\mathfrak{m}}

\newcommand{\inj}{\hookrightarrow}
\newcommand{\surj}{\twoheadrightarrow}

\newcommand{\Ind}{\mathrm{Ind}}

\newcommand{\triv}{\mathrm{triv}}
\newcommand{\Res}{\mathrm{Res}}

\newcommand{\SL}{\mathrm{SL}}
\newcommand{\GL}{\mathrm{GL}}
\newcommand{\PSL}{\mathrm{PSL}}
\newcommand{\actson}{\curvearrowright}

\newcommand{\mat}[4]{\left(\begin{array}{cc} #1 & #2\\ #3 & #4\end{array}\right)}
\newcommand{\twobytwo}[4]{\mat{#1}{#2}{#3}{#4}}

\newcommand{\iso}{\cong}
\newcommand{\eps}{\varepsilon}

\newcommand{\Frob}{\mathrm{Frob}}

\newcommand{\afrak}{\mathfrak{a}}

\newcommand{\ofrak}{\mathfrak{o}}
\newcommand{\A}{\mathbb{A}}
\newcommand{\I}{\mathbb{I}}

\newcommand{\G}{\mathbb{G}}

\renewcommand{\o}{\ofrak}

\newcommand{\Cl}{\mathrm{Cl}}

\newcommand{\Qbar}{\overline{\Q}}

\renewcommand{\P}{\mathbb{P}}

\renewcommand{\I}{\mathrm{I}}

\newcommand{\tr}{\mathrm{tr}}
\newcommand{\diag}{\mathrm{diag}}

\newcommand{\Fbar}{\overline{\F}}
\newcommand{\nfrak}{\mathfrak{n}}

\newcommand{\rhobar}{\overline{\rho}}

\newcommand{\chibar}{\overline{\chi}}
\newcommand{\CM}{\text{CM}}
\renewcommand{\End}{\text{End}}
\newcommand{\opp}{\text{opp.}}
\newcommand{\BT}{\text{BT}}

\renewcommand{\to}{\longrightarrow}
\renewcommand{\inj}{\lhook\joinrel\longrightarrow}
\renewcommand{\surj}{\longrightarrow\mathrel{\mkern-25mu}\longrightarrow}
\newcommand{\shortinj}{\hookrightarrow}

\let\phi\varphi

\let\emptyset\varnothing

\makeatletter
\let\@@pmod\pmod
\DeclareRobustCommand{\pmod}{\@ifstar\@pmods\@@pmod}
\def\@pmods#1{\mkern4mu({\operator@font mod}\mkern 6mu#1)}
\makeatother

 \DeclareFontFamily{U}{wncy}{}
    \DeclareFontShape{U}{wncy}{m}{n}{<->wncyr10}{}
    \DeclareSymbolFont{mcy}{U}{wncy}{m}{n}
    \DeclareMathSymbol{\Sha}{\mathord}{mcy}{"58}

\renewcommand{\to}{\longrightarrow}
\renewcommand{\inj}{\lhook\joinrel\longrightarrow}
\renewcommand{\surj}{\longrightarrow\mathrel{\mkern-25mu}\longrightarrow}

\title{Un peu d'effectivit\'{e} pour les vari\'{e}t\'{e}s modulaires de Hilbert-Blumenthal.}

\author{\Large Levent Alp\"{o}ge}

\date{}

\begin{document}

\maketitle

\renewcommand{\abstractname}{Abstract.}
\begin{abstract}
We prove a "height-free" effective isogeny estimate for abelian varieties of $\GL_2$-type.

More precisely, let $g\in \Z^+$, $K$ a number field, $S$ a finite set of places of $K$, and $A,B/K$ $g$-dimensional abelian varieties with good reduction outside $S$ which are $K$-isogenous and of $\GL_2$-type over $\Qbar$. We show that there is a $K$-isogeny $A\rightarrow B$ of degree effectively bounded in terms of $g$, $K$, and $S$ only.

We deduce among other things an effective upper bound on the number of $S$-integral $K$-points on a Hilbert modular variety.
\end{abstract}

\tableofcontents

\section{Introduction.}
The crux of Faltings' Great \cite{faltings} is his proof of the finiteness of isogeny classes of abelian varieties over number fields, from which he deduces a finiteness conjecture of Shafarevich. His proof does not\footnote{Experience with the numerous ineffective theorems of number theory leads one to expect that, in any particular ineffective argument, even if one cannot give a finite-time algorithm to determine the exceptional set that one proves is finite, one should at least be able to give an effective upper bound for its size. However this is not the case for Faltings' proof of the Shafarevich conjecture. On the other hand it \emph{is} the case for his proof of the Mordell conjecture, thanks to Raynaud's effectivization of his isogeny theorem and the very clever idea of Parshin to invoke Mumford's gap principle --- see Szpiro's \cite{szpiro}. There are now of course numerous other ways to proceed.} give an estimate on the size $\#|\mathcal{A}_g(\o_{K,S})|$ of the relevant finite set, and nor do the subsequent effectivizations and improvements of Raynaud, Masser-W\"{u}stholz, Bost, David, Gaudron-R\'{e}mond, and others\footnote{Let us also note work of the Chudnovsky brothers \cite{chudnovsky-brothers} and Koshikawa \cite{koshikawa}. Our list certainly omits important work, and we ask the reader to forgive our ignorance of the literature.}. This is because all known\footnote{We exclude the case of elliptic curves because it is trivialized by Baker's explicit height bounds (since the relevant moduli space is so simple).} proofs of his isogeny theorem give bounds depending on the height of a representative in the relevant isogeny class, a quantity over which one still has no a priori effective control.

The purpose of this paper is to prove a "height-free" isogeny estimate for a special class of abelian varieties, namely those of $\GL_2$-type.

(We have already explained elsewhere why sufficiently accurate study of these abelian varieties suffices for an effective form of the Mordell conjecture for e.g.\ hyperelliptic curves.) We will give an a priori effective isogeny estimate depending only on their dimensions, conductors, and fields of definition. The key point in the argument will be provided by a lemma which greatly simplified the argument in \cite{my-first-effective-mordell-paper}.\footnote{Before writing \cite{my-thesis} it was clear to us that this lemma must be known at least with ineffective implied constant (indeed it is a special case of a theorem of Zarhin which relies on Faltings' finiteness of isogeny classes). However in the context of \cite{my-first-effective-mordell-paper} the lemma can also be phrased as a statement about congruences of Hilbert modular forms, whence it was natural to look in the literature on Hilbert modular forms for such a statement. And exactly such a statement was proven by Mladen Dimitrov \cite{dimitrov} via a technique that we effectivized with a trick. This paper then amounts to "reversing the steps" and using said lemma to deduce something about Faltings' finiteness of isogeny classes.}

Such an isogeny estimate easily gives an a priori estimate on the number of such abelian varieties of bounded dimension, conductor, and field of definition (and their endomorphism rings, and e.g.\ an effective adelic surjectivity statement \`{a} la Serre), as we will explain.

Let us now list the main results in this paper.

Theorem \ref{a priori isogeny estimate for GL2-type abelian varieties} constitutes our "height-free" isogeny estimate for abelian varieties of $\GL_2$-type over a number field.

The resulting effective upper bound for the number of (isomorphism classes of) such abelian varieties is Theorem \ref{upper bound on the number of GL2-type abelian varieties}, from which we of course produce an effective upper bound on the number of $S$-integral points on moduli spaces of such abelian varieties --- see Corollary \ref{upper bound on the number of S-integral K-points on a Hilbert modular variety}.

From said upper bound we deduce effective upper bounds on the number of rational points on a large class of curves (including e.g.\ all hyperelliptic curves over number fields) --- see Corollary \ref{un peu d'effectivite for sirin curves}. Of course by now this is known in a number of ways, and our only interest in this is that our argument follows Szpiro's \cite{szpiro} but avoids Parshin's very clever idea of using Mumford's gap principle by instead invoking our isogeny estimate.

We also prove that the (isomorphism classes of) endomorphism rings of $g$-dimensional abelian varieties of $\GL_2$-type over a number field $K$ and with good reduction outside a finite set $S$ of places of $K$ lie in an effectively computable finite set of possibilities --- see Theorem \ref{the endomorphism ring is one of an explicit finite set of possibilities}.

Finally our adelic surjectivity theorem is Theorem \ref{serre open image theorem for GL2-type abelian varieties}, we effectivize a theorem of M.\ Dimitrov on the residual representations associated to Hilbert modular eigencuspforms in Corollary \ref{an explicit form of a theorem of dimitrov}, and we explain in Appendix \ref{the appendix} how to prove the finiteness of isogeny classes of elliptic curves and fake elliptic curves by reasoning only about Galois representations.

\subsection{Main theorems.}

We prove the following.

\begin{thm}\label{a priori isogeny estimate for GL2-type abelian varieties}
Let $g\in \Z^+$. Let $K/\Q$ be a number field. Let $S$ be a finite set of places of $K$. Then: there is an effectively computable constant $\text{\c{s}}_{g,K,S}\in \Z^+$ depending only on $g$, $K$, and $S$ such that the following holds.
\begin{itemize}
\item Let $A,B/K$ be $g$-dimensional abelian varieties over $K$ which are $K$-isogenous, have good reduction outside $S$, and are of $\GL_2$-type over $\Qbar$. Then: $\Hom_K(A,B)$ is generated as an abelian group by $K$-isogenies of degree $\leq \text{\c{s}}_{g,K,S}$.
\end{itemize}
\end{thm}\noindent
Our argument is in principle completely explicit except for one step where we avoid a calculation with Breuil-Kisin modules via a trick. In other words we believe there is no obstruction to making the map $(g,K,S)\mapsto \textit{\c{s}}_{g,K,S}$ explicit.

This easily implies the following two theorems.
\begin{thm}\label{the endomorphism ring is one of an explicit finite set of possibilities}
Let $g\in \Z^+$. Let $K/\Q$ be a number field. Let $S$ be a finite set of places of $K$. Then: there is an effectively computable finite set of (isomorphism classes of) rings $\mathcal{R}_{g,K,S}$ depending only on $g$, $K$, and $S$ such that the following holds.
\begin{itemize}
\item Let $A/K$ be a $g$-dimensional abelian variety of $\GL_2$-type over $\Qbar$ which has good reduction outside $S$. Then: $\End_K(A)\in \mathcal{R}_{g,K,S}$.
\end{itemize}
\end{thm}

\begin{thm}\label{upper bound on the number of GL2-type abelian varieties}
Let $g\in \Z^+$. Let $K/\Q$ be a number field. Let $S$ be a finite set of places of $K$. Then: there is an effectively computable constant $\bigiy_{g,K,S}\in \Z^+$ depending only on $g$, $K$, and $S$ such that the number of $K$-isomorphism classes of abelian varieties $A/K$ which are of $\GL_2$-type over $\Qbar$ and have good reduction outside $S$ is $\leq \bigiy_{g,K,S}$.
\end{thm}

Of course the following follows immediately.

\begin{cor}\label{upper bound on the number of S-integral K-points on a Hilbert modular variety}
Let $g\in \Z^+$. Let $F/\Q$ be a totally real field of degree $g = [F:\Q]$. Let $\o\subseteq \o_F$ be an order in $\o_F$. Let $K/\Q$ be a number field. Let $S$ be a finite set of places of $K$.\footnote{To ignore subtleties with integral models we implicitly add to $S$ all primes of $K$ of norm $\leq \left(10^{10}\cdot |\Delta_{\o}|\right)^{10^{10}\cdot [K:\Q]}$.} Then: there is an effectively computable constant $\bighardsign_{g,K,S}\in \Z^+$ depending only on $g$, $K$, and $S$ such that the number of $\o_{K,S}$-points on the canonical $\o_{K,S}$-integral model $\mathcal{H}_\o$ of the\footnote{By this we really mean that the statement holds regardless of the choice of polarization data, which our notation leaves implicit.} Hilbert modular stack associated to $\o$ is $\#|\mathcal{H}_\o(\o_{K,S})| \leq \bighardsign_{g,K,S}$.
\end{cor}

We also deduce an effective open image theorem \`{a} la Serre for $\GL_2$-type abelian varieties. For $F/\Q$ a number field and $\o\subseteq F$ an order, write, informally, $G_\o := \{g\in \Res_{\o/\Z}\,{\GL_2} : \det{g}\in \G_m\subseteq \Res_{\o/\Z}\,{\G_m}\}$, so that $G_{\o}(\Z_p) := \{g\in \GL_2(\o_p) : \det{g}\in \Z_p^\times\}$, where as usual we have written $\o_p := \o\otimes_\Z \Z_p$.

\begin{thm}\label{serre open image theorem for GL2-type abelian varieties}
Let $g\in \Z^+$. Let $K/\Q$ be a number field. Let $S$ be a finite set of places of $K$. Then: there is an explicit (thus effectively computable) constant $\emph{\text{\^{\i}}}_{g,K,S}\in \Z^+$ depending only on $g$, $K$, and $S$ such that the following holds.

\begin{itemize}
\item Let $A/K$ be a $g$-dimensional $\Qbar$-simple split semistable abelian variety over $K$ which has good reduction outside $S$, is of $\GL_2$-type over $K$, does not admit sufficiently many complex multiplications over $\Qbar$, and is such that, for all $p\geq \emph{\text{\^{\i}}}_{g,K,S}$, $\det{\rho_{A,p}} = \chi_p$, where $\chi_p$ is the $p$-adic cyclotomic character and, writing $\o$ for the centre of $\End_K(A)$ (thus $\o\subseteq F$ is an order in a CM\footnote{For us a CM field $K/\Q$ is either totally real or else a totally imaginary quadratic extension of a totally real field.} $F/\Q$ with $[F:\Q] = g$ or $\frac{g}{2}$), $\rho_{A,p} : \Gal(\Qbar/K)\to \GL_2(\o_p)$ is the representation corresponding to the $\End_K(A)[\Gal(\Qbar/K)]$-module structure of $T_p(A)$.

Then: $\left(\prod_{p\geq \emph{\text{\^{\i}}}_{g,K,S}} \rho_{A,p}\right)(\Gal(\Qbar/K)) = \prod_{p\geq \emph{\text{\^{\i}}}_{g,K,S}} G_\o(\Z_p)$.
\end{itemize}
\end{thm}\noindent
For small primes we content ourselves with Proposition \ref{small prime lemma}. Of course, given a $\Qbar$-simple $g$-dimensional abelian variety $A/K$ which is of $\GL_2$-type over $\Qbar$ and which has good reduction outside $S$, there is an explicit finite extension $L/K$ depending only on $g$, $K$, and $S$ such that the hypotheses of Theorem \ref{serre open image theorem for GL2-type abelian varieties} hold for $A/L$.

It is perhaps worth noting that for relevant abelian varieties with compact moduli Theorem \ref{serre open image theorem for GL2-type abelian varieties} gives a uniform effective adelic surjectivity statement --- e.g.\ for fake elliptic curves one has the following.
\begin{cor}\label{serre open image theorem for fake elliptic curves}
Let $B/\Q$ be a quaternion algebra over $\Q$ of discriminant $\Delta_B\neq 1$. Let $\o\subseteq B$ be a maximal order. Let $X_{\o}/\Q$ be the corresponding (fine) Shimura curve. Let $K/\Q$ be a number field. Then: there is an explicit (thus effectively computable) constant $\zh_{B,K}\in \Z^+$ depending only on $B$ and $K$ such that the following holds.

\begin{itemize}
\item Let $P\in X_{\o}(K)$ be non-CM. Let $A/K$ be the ($\Qbar$-simple) abelian surface corresponding to $P$. Let, for $p\nmid \Delta_B$, $\rho_{A,p} : \Gal(\Qbar/K)\to \GL_2(\Z_p)$ be the representation corresponding to the $\o[\Gal(\Qbar/K)]$-module structure of $T_p(A)$. Then: $\left(\prod_{p\geq \zh_{B,K}} \rho_{A,p}\right)(\Gal(\Qbar/K)) = \prod_{p\geq \zh_{B,K}} \GL_2(\Z_p)$.
\end{itemize}
\end{cor}\noindent
Note e.g.\ that by Theorem \ref{the endomorphism ring is one of an explicit finite set of possibilities} the $\o$ for which there is an abelian surface $A/K$ with good reduction everywhere and $\End_K(A)\iso \o$ lie in an explicit finite set.

We conclude with the following two observations.

For the first observation we introduce the following abbreviation.

\begin{defn}
Let $K/\Q$ be a number field. Let $C/K$ be a smooth projective hyperbolic curve over $K$. Then: $C/K$ is \emph{\c{s}irin} if and only if there is a finite extension $L/K$, a smooth projective hyperbolic curve $\widetilde{C}/L$, an \'{e}tale cover $\phi: \widetilde{C}\to C$ defined over $L$, and a nonisotrivial family $\pi: A\to \widetilde{C}$ of abelian varieties which are of $\GL_2$-type over $L(\widetilde{C})$.

We also call the data $(C, K, L, \widetilde{C}, \phi, \pi)$ a \emph{\c{s}irin family}.
\end{defn}\noindent
(For example, all solvable covers of $\P^1$ over a number field are \c{s}irin, via e.g.\ the family given in Section $6$ of \cite{my-first-effective-mordell-paper} and the main theorem of Poonen's \cite{poonen}. More generally, all curves $C/K$ admitting a diagram $C\xleftarrow{\phi} \widetilde{C}\xrightarrow{f} \P^1$ over a finite $L/K$ with $\phi$ \'{e}tale and $f$ Belyi with all ramification indices above $0$, $1$, $\infty$  respectively divisible by $a$, $b$, $c$ with $\frac{1}{a} + \frac{1}{b} + \frac{1}{c} < 1$ are also \c{s}irin, by pulling back a suitable hypergeometric family of abelian varieties.)

Now for the first observation.

\begin{cor}\label{un peu d'effectivite for sirin curves}
Let $K/\Q$ be a number field. Let $C/K$ be a \c{s}irin smooth projective hyperbolic curve over $K$. Then: there is an effectively computable constant $\I_{K,C}\in \Z^+$ depending only on $K$ and $C/K$ such that $\#|C(K)|\leq \I_{K,C}$.
\end{cor}
Of course this corollary is not new, and we have already explained our interest. (We refer the reader to V.\ Dimitrov-Gao-Habegger's \cite{dimitrov-gao-habegger} for far stronger bounds, as well as to Chapter $6$ of \cite{my-thesis} for a bound on the number of large points only.)

Finally let us state the second observation. Write, for $f$ a Hilbert modular eigencuspform of weight $\vec{k}$ and level $\nfrak$ over a totally real field $F$ and $\ell\gg_{F,\vec{k},\nfrak} 1$ a prime of $\Z$, $\rho_{f,\ell}: \Gal(\Qbar/F)\to \GL_2(\Qbar_\ell)$ for the corresponding $\ell$-adic representation. Write also $\rhobar_{f,\ell}: \Gal(\Qbar/F)\to \GL_2(\Fbar_\ell)$ for the (semisimplified) residual representation associated to $\rho_{f,\ell}$. We will use M.\ Dimitrov's abbreviations $(\textbf{Irr}_{\rhobar})$, $(\textbf{LI}_{\rhobar})$, and $(\textbf{LI}_{\Ind\,{\rhobar}})$ --- see Proposition $0.1$ of Dimitrov's \cite{dimitrov} for definitions. (We will define all but $(\textbf{LI}_{\Ind\,{\rhobar}})$ below.)

\begin{cor}[cf.\ Proposition $0.1$ of M.\ Dimitrov's \cite{dimitrov}]\label{an explicit form of a theorem of dimitrov}
Let $F/\Q$ be totally real. Let $\nfrak\subseteq \o_F$ be an ideal of $\o_F$. Let $S_\infty$ be the finite set of infinite places of $F$. Let $\vec{k}\in \Z^{S_\infty}$. Then: there is an effectively computable constant $\dot{\I}_{F,\vec{k},\nfrak}\in \Z^+$ depending only on $F$, $\vec{k}$, and $\nfrak$ such that the following holds.
\begin{itemize}
\item Let $p\geq \dot{\I}_{F,\vec{k},\nfrak}$ be a prime. Let $f$ be a Hilbert modular eigencuspform over $F$ of weight $\vec{k}$ and level $\nfrak$. Then: $\rhobar_{f,p}$ is absolutely irreducible, i.e.\ Dimitrov's condition $(\textbf{\emph{Irr}}_{\rhobar})$ holds.
\item Let $p\geq \dot{\I}_{F,\vec{k},\nfrak}$ be a prime. Let $f$ be a Hilbert modular eigencuspform over $F$ of weight $\vec{k}$ and level $\nfrak$ which is not a theta series. Then: there is a power $q$ of $p$ and a $g\in \GL_2(\Fbar_p)$ such that $g\cdot \SL_2(\F_q)\cdot g^{-1}\subseteq \rhobar_{f,\ell}(\Gal(\Qbar/F))\subseteq \Fbar_\ell^\times\cdot (g\cdot \GL_2(\F_q)\cdot g^{-1})$, i.e.\ Dimitrov's condition $(\textbf{\emph{LI}}_{\rhobar})$ holds.
\item Let $p\geq \dot{\I}_{F,\vec{k},\nfrak}$ be a prime. Let $f$ be a Hilbert modular eigencuspform over $F$ of weight $\vec{k}$ and level $\nfrak$ which is not a theta series and which is not a twist by a character of any of its internal conjugates. Then: Dimitrov's condition $(\textbf{\emph{LI}}_{\Ind\,{\rhobar}})$ holds.
\end{itemize}
\end{cor}

\subsection{Technique.}
We will only describe the argument proving the isogeny estimate. We will moreover ignore small primes in our sketch.

The first matter is to control a priori the endomorphism algebra of a given $A/K$ of $\GL_2$-type over $K$ with good reduction outside $S$. Now, $A/K$ is either isotypic or else $K$-isogenous to the product of two isotypic abelian varieties which are CM over $\Qbar$. Let us ignore the CM case. Via a Serre tensor product it suffices to treat the case where $A = B^{\times n}$ with $B/K$ $K$-simple and with geometric endomorphisms by a maximal order in $\End_K^0(B)$. There are two cases: either $\End_K^0(B) = F$ a CM field of degree $g$, or else $\End_K^0(B) = D$ a quaternion algebra over a CM field of degree $\frac{g}{2}$.

Let us treat the first case for a moment. It is an easy matter to use Faltings' proof of the Tate conjecture to prove that $F$ is generated over $\Q$ by an explicit finite set of Frobenius traces (at a Faltings-Serre set of primes associated to an auxiliary prime, just as in our proof of Lemma $3.9$ in our \cite{my-first-effective-mordell-paper}). This of course upper bounds the discriminant of $F$ and thus restricts it to an explicit finite set by Hermite-Minkowski.

(In the quaternionic multiplication case we upper bound the centre in exactly the same way and then upper bound the discriminant of the quaternion algebra $D$ by again using Lemma $3.9$ of \cite{my-first-effective-mordell-paper} --- were $D$ ramified above a large prime $p$, one of the mod-$p$ residual representations of $B/K$ would automatically be reducible.)

So this controls the relevant endomorphism algebras a priori.

So now let $A\to C$ be a $K$-isogeny, let $G$ be its kernel, and write $G =: \bigoplus_p G_p$ with $G_p\subseteq A[p^\infty]$. It is natural to try to write $G = \bigoplus_\pfrak G_\pfrak$, but there arises the following scare: we seemingly have no control over $\End_K(C)$, so why should $G$ be stable under $\End_K(B) = \o_F$?

To overcome this we use the following trick: for $g\in \GL_2$, $$g + (\det{g})\cdot g^{-1} = (\tr\,{g})\cdot \id.$$ It is an easy matter to make the determinants of the relevant $2$-dimensional Galois representations cyclotomic on the nose (by just passing to an explicit extension provided by Hermite-Minkowski), so in e.g.\ the case of $\End_K^0(B) = F$, we have that, for all $p$, $$\Z_p[\rho_{B,p}(\Gal(\Qbar/K))]\supseteq \Z[\{\tr(\rho_{B,p}(\Frob_\qfrak)) : \qfrak\in T\}]\otimes_\Z \Z_p,$$ where $\rho_{B,p}: \Gal(\Qbar/K)\to \GL_2(\o_{F,p})$ and $T$ is a Faltings-Serre set of primes associated to a suitably chosen auxiliary prime. Because of this choice of $T$ we know that $\Z[\{\tr(\rho_{B,p}(\Frob_\qfrak)) : \qfrak\in T\}]\subseteq F$ is an order. But said order is manifestly generated by elements of bounded height, whence it has bounded discriminant, whence bounded index in $\o_F$, whence for $p$ sufficiently large we find that $$\Z_p[\rho_{B,p}(\Gal(\Qbar/K))]\supseteq \o_{F,p}.$$

Thus $\Z_p[\rho_{B,p}(\Gal(\Qbar/K))] = \o_{F,p}[\rho_{B,p}(\Gal(\Qbar/K))] = \bigoplus_{\pfrak\vert (p)} \o_{F,\pfrak}[\rho_{B,\pfrak}(\Gal(\Qbar/K))]$.

Now Lemma $3.9$ of \cite{my-first-effective-mordell-paper} gives more than just residual irreducibility --- in fact for all $\pfrak\vert (p)$ the residual image of $\rho_{B,\pfrak}$ contains a conjugate of $\SL_2(\F_p)$. Hence by Nakayama we find that $\o_{F,\pfrak}[\rho_{B,\pfrak}(\Gal(\Qbar/K))] = M_2(\o_{F,\pfrak})$, whence $$\Z_p[\rho_{B,p}(\Gal(\Qbar/K))] = M_2(\o_{F,p}).$$

This means that just from Galois-invariance we automatically get that (for $p$ large) $G_p$ is $M_2(\o_{F,p})$-invariant.

Let $N_p\in \Z^+$ with $G_p\subseteq A[p^{N_p}]$ and let $\Gamma_p\subseteq T_p(A) = T_p(B)^{\oplus n}$ be the preimage of $G_p\subseteq A[p^{N_p}]\simeq T_p(A)/p^{N_p}$. This means that $\Gamma_p\inj T_p(A)\iso (\o_{F,p}^{\oplus 2})^{\oplus n}$ is an $M_2(\o_{F,p})$-submodule (diagonal action). So by Morita it follows that $\Gamma_p = \widetilde{\Gamma}_p^{\oplus 2}$ with $\widetilde{\Gamma}_p\subseteq \o_{F,p}^{\oplus n}$ an $\o_{F,p}$-submodule.

But because $\o_{F,p}$ is a direct sum of principal ideal domains, it follows that there is an $\alpha_p\in M_n(\o_{F,p})$ such that $\widetilde{\Gamma}_p = \alpha_p\cdot \o_{F,p}^{\oplus n}$, whence $\Gamma_p = \alpha_p\cdot T_p(A)$. Since $p^{N_p}\cdot T_p(A)\subseteq \Gamma_p$ it follows that there is a $\beta_p\in M_n(\o_{F,p})$ such that $\beta_p\cdot \alpha_p = p^{N_p}\cdot \id$, and then $G_p = \ker(\beta_p\actson A[p^\infty])$. Of course (by shifting $\alpha_p$ suitably) we may equally well replace $\beta_p$ by any element of $\GL_n(\o_{F,p})\cdot \beta_p$.

By considering the class $(\beta_p)_p$ inside $\left(\prod_p \GL_n(\o_{F,p})\right)\backslash \GL_n(\A_F^{\text{fin.}})/\GL_n(F)$ and applying Minkowski we conclude that there is a $\gamma\in M_n(\o) = \End_K(A)$ such that $\ker{\gamma}\supseteq G$ and $[\ker{\gamma} : G]\ll_{g,K,S} 1$.

This exactly says that there is a $K$-isogeny $C\to A$ of degree bounded by $\ll_{g,K,S} 1$, namely the $K$-isogeny corresponding to the $K$-subgroup $(\ker{\gamma})/G\subseteq A/G = C$ --- QED isogeny estimate.

When $\End_K(B)$ is a maximal order $\o$ in a quaternion algebra the same argument goes through with minor modifications (namely by replacing $M_2(\o_{F,p})$ with $\o_p := \o\otimes_\Z \Z_p$ throughout --- this also explains why we unnecessarily considered the adelic double coset space in the case $\End_K^0(B) = F$).

\subsection{Remarks.}

Let us quickly explain why the isogeny estimates of Raynaud/Faltings and Masser-W\"{u}stholz (as made explicit by Bost, David, and Gaudron-R\'{e}mond) do not suffice to prove Theorem \ref{upper bound on the number of GL2-type abelian varieties} (or Theorem \ref{the endomorphism ring is one of an explicit finite set of possibilities} etc.).

A coarse version (that is, ignoring the explicit dependence of the constant which was also given by Raynaud) of Raynaud's effectivization of Faltings' isogeny theorem is the following.\footnote{Koshikawa's aforementioned improvement --- i.e.\ Theorem $9.8$ of his \cite{koshikawa} --- amounts, in these coarse terms, to removing the dependence of $C_{g,K,S}$ on $S$.}
\begin{thm}[Raynaud --- cf.\ Th\'{e}or\`{e}me $4.4.9$ of \cite{raynaud}]\label{raynaud's isogeny estimate}
Let $g\in \Z^+$. Let $K/\Q$ be a number field. Let $S$ be a finite set of places of $K$. Then: there is an explicit (thus effectively computable) constant $C_{g,K,S}\in \Z^+$ depending only on $g$, $K$, and $S$ such that the following holds.
\begin{itemize}
\item Let $A,B/K$ be $g$-dimensional abelian varieties over $K$ which are $K$-isogenous and have good reduction outside $S$. Then: $|h(A) - h(B)|\leq C_{g,K,S}$.
\end{itemize}
\end{thm}\noindent
One concludes finiteness by an application of Northcott. But the resulting estimate on the size of the finite set of abelian varieties $B/K$ which are $K$-isogenous to $A/K$ \emph{depends on $h(A)$}. This is why Szpiro needs to apply Mumford's gap principle in his \cite{szpiro} --- regardless of the interval, at least there is a uniform (here meaning only independent of $h(A)$) bound on the number of points on a hyperbolic \emph{curve} whose heights lie in an interval of given length.

This would not be an issue if we had an effective estimate $h(A)\ll_{g,K,S} 1$ available --- but of course if we had this then we would not be asking the question, since such an estimate would amount to a solution of the effective Shafarevich conjecture in this case, a problem which is open (though see \cite{my-first-effective-mordell-paper}).

Now let us turn to the Masser-W\"{u}stholz isogeny estimate, as improved and made explicit by Gaudron-R\'{e}mond (after work of Bost and David). Again, we will state a coarse version for the sake of explanation.
\begin{thm}[Gaudron-R\'{e}mond --- cf.\ Th\'{e}or\`{e}me $1.4$ of \cite{gaudron-remond}]\label{masser-wustholz}
Let $g\in \Z^+$. Let $K/\Q$ be a number field. Let $H\in \R^+$. Then: there is an explicit (thus effectively computable) constant $C_{g,K,H}\in \Z^+$ depending only on $g$, $K$, and $H$ such that the following holds.
\begin{itemize}
\item Let $A/K$ be a $g$-dimensional abelian variety over $K$ with $h(A)\leq H$. Let $B/K$ be $K$-isogenous to $A/K$. Then: there is a $K$-isogeny $\phi: A\to B$ with $\deg{\phi}\leq C_{g,K,H}$.
\end{itemize}
\end{thm}\noindent
In fact one can take e.g.\ $C_{g,K,H} := \max\left(10^{10}\cdot [K:\Q], H\right)^{O_g(1)}$, to be a bit more explicit.

One concludes finiteness because there are a finite number of $K$-subgroups of $A$ of size at most $C_{g,K,h(A)}$, but again the resulting bound depends on $h(A)$, and of course this is the same issue that arose in trying to use Raynaud's isogeny theorem.

Having explained this point, let us now prove the theorems.

\section{Acknowledgements.}
Lemma \ref{irreducibility after some point} is based on Lemma $9.3.3$ of Chapter $9$ of the author's Ph.D. thesis at Princeton University \cite{my-thesis}. I would like to thank both my advisor Manjul Bhargava and Peter Sarnak for their patience and encouragement. I would also like to thank Jacob Tsimerman and Nina Zubrilina for informative discussions. Finally I thank the National Science Foundation (via their grant DMS-$2002109$), Columbia University, and the Society of Fellows for their support during the pandemic.

\section{Results from our \cite{my-first-effective-mordell-paper}.}

We will cite four results from our \cite{my-first-effective-mordell-paper}.

\subsection{Our reducibility estimate.\label{irreducibility after some point section}}

The first is Lemma $3.9$ of our \cite{my-first-effective-mordell-paper}. We will state it here in the particular case of abelian varieties only.

\begin{lem}[Cf.\ Lemma $3.9$ of \cite{my-first-effective-mordell-paper}.]\label{irreducibility after some point}
Let $d\in \Z^+$. Let $K/\Q$ be a number field. Let $N\in \Z^+$.

Then: there is an explicit (thus effectively computable) constant $C_{d,K,N}\in \Z^+$ depending explicitly and only on $d$, $K$, and $N$ such that the following holds.
\begin{itemize}
\item Let $p\geq C_{d,K,N}$ be a prime of $\Z$. Let $E/\Q$ be a number field of degree $[E:\Q]\leq d$. Let $\pfrak\subseteq \o_E$ be a prime of $\o_E$ with $\pfrak\vert (p)$. Let $A/K$ be an abelian variety of $\GL_2(E)$-type over $K$ which does not admit sufficiently many complex multiplications over $\Qbar$ and has good reduction outside the primes of $K$ dividing $(N)$. Then: writing the mod-$\pfrak$ residual representation of $A/K$ as $\overline{\rho}_{A,\pfrak} : \Gal(\Qbar/K)\to \GL_2(\o_E/\pfrak)$, there is a $g\in \GL_2(\o_E/\pfrak)$ and a subfield $\F_q\subseteq \o_E/\pfrak$ such that: $$g\cdot \SL_2(\F_q)\cdot g^{-1}\subseteq \overline{\rho}_{A,\pfrak}(\Gal(\Qbar/K))\subseteq (\o_E/\pfrak)^\times\cdot (g\cdot \GL_2(\F_q)\cdot g^{-1}).$$
\end{itemize}
\end{lem}

In fact we will prove the following stronger statement in Section \ref{proof of serre open image theorem section}.

\begin{thm}\label{huge image after some point}
Let $g\in \Z^+$. Let $K/\Q$ be a number field. Let $N\in \Z^+$.

Then: there is an explicit (thus effectively computable) constant $C_{g,K,N}\in \Z^+$ depending explicitly and only on $g$, $K$, and $N$ such that the following holds.
\begin{itemize}
\item Let $p\geq C_{g,K,N}$ be a prime of $\Z$. Let $A/K$ be a $\Qbar$-simple abelian variety of $\GL_2$-type over $K$ which does not admit sufficiently many complex multiplications over $\Qbar$ and has good reduction outside the primes of $K$ dividing $(N)$. Let $E$ be the centre of $\End_K^0(A)$ (thus $E/\Q$ is a CM field of degree either $g$ or $\frac{g}{2}$). Let $\pfrak\subseteq \o_E$ be a prime of $\o_E$ with $\pfrak\vert (p)$. Then: writing the mod-$\pfrak$ residual representation of $A/K$ as $\overline{\rho}_{A,\pfrak} : \Gal(\Qbar/K)\to \GL_2(\o_E/\pfrak)$, $$\rhobar_{A,\pfrak}(\Gal(\Qbar/K))\supseteq \SL_2(\o_E/\pfrak).$$
\end{itemize}
\end{thm}

For completeness let us now sketch a proof of Lemma \ref{irreducibility after some point} (by sketching the proof of Lemma $3.9$ of \cite{my-first-effective-mordell-paper}).

\begin{proof}[Sketch of proof of Lemma \ref{irreducibility after some point}.]
By the Dickson classification of subgroups of $\PSL_2(\o_E/\pfrak)$ (see e.g.\ Theorem $4.15$ of \cite{flannery-o'brien}) and the fact that $\SL_2(\F_q)$ is its own commutator subgroup, to show that $\SL_2(\F_q)\subseteq g\cdot \rhobar_{A,\pfrak}(\Gal(\Qbar/K))\cdot g^{-1}\subseteq (\o_E/\pfrak)^\times\cdot \GL_2(\F_q)$ for some $g\in \GL_2(\o_E/\pfrak)$ and some subfield $\F_q\subseteq \o_E/\pfrak$ it suffices to show that $\rhobar_{A,\pfrak}$ is absolutely irreducible, has projective image of size $\gg 1$, and is absolutely not induced. To see that the projective image is of size $\gg 1$ one restricts to inertia subgroups at $\qfrak\vert (p)$ and uses Raynaud's classification of the inertial restrictions of Galois representations corresponding to finite flat group schemes which prolong over $\o_{K,\qfrak}$ (thus $(\rhobar_{A,\pfrak}\otimes_{\o_E/\pfrak} \Fbar_p)\vert_{I_\qfrak}$ semisimplifies to a sum of two multiplicity-free products of fundamental characters, and since $\det{(\rhobar_{A,\pfrak}\vert_{I_\qfrak})}$ is cyclotomic it follows that the projectivization of $\rhobar_{A,\pfrak}$ must have large image).

Showing that $\rhobar_{A,\pfrak}$ is absolutely not induced reduces to showing that $\rhobar_{A,\pfrak}\vert_{\Gal(\Qbar/L)}$ is absolutely irreducible for $L/K$ an explicit finite extension, so let us focus on absolute irreducibility. Let us just show irreducibility for notation's sake.

Suppose $\rhobar_{A,\pfrak}$ is reducible. Thus it semisimplifies to a sum of characters, say $\tr{\rho_{A,\pfrak}}\equiv \chibar + \chibar'\pmod*{\pfrak}$. Since $\chibar$ and $\chibar'$ are subquotients of the representation $\rhobar_{A,\pfrak}$, and the latter corresponds to a finite flat group scheme which prolongs over $\o_{K,S}$, it follows that the inertial restrictions $\chibar\vert_{I_\qfrak}$ and $\chibar'\vert_{I_\qfrak}$ at primes $\qfrak\vert (p)$ are reductions modulo $\pfrak$ of characters corresponding to CM $p$-divisible groups.

Because (via Artin reciprocity) $\chibar$ and $\chibar'$ must be trivial on an explicit finite-index subgroup of global units, in fact said CM $p$-divisible groups must be $p$-divisible groups associated to algebraic Hecke characters --- indeed, abusing notation by leaving the Artin map implicit, $\chibar(\eps)\equiv 1\pmod*{\pfrak}$ for all totally positive $\eps\in \o_K^\times$ with $\eps\equiv 1\pmod*{(N)^{(10^{10}\cdot d\cdot [K:\Q])!}}$ means (from the fact that $\chibar$ lifts to a CM $p$-divisible group at each inertia group) that a particular product of embeddings of $\eps$ into a Galois closure $L$ of $K$ is $1$ modulo a prime above $\pfrak$. But, considering this congruence only on generators of this finite-index subgroup of $\o_K^\times$, such a product has explicitly bounded height, and $p$ is huge, so we must have equality, not just congruence --- at least on the generators. Multiplicativity implies that said product of embeddings is trivial on a finite-index subgroup of $\o_K^\times$, and so, by the classification of algebraic Hecke characters, it follows that there was in fact an algebraic Hecke character $\psi$ such that $\chibar\vert_{I_\qfrak}\equiv \psi\vert_{I_\qfrak}\pmod*{\pfrak}$. Similarly for $\chibar'$.

Twisting by $\psi^{-1}$ leaves a finite-order character with conductor dividing an explicit constant depending on $N$, so that we conclude that there are algebraic Hecke characters $\chi$ and $\chi'$ of explicitly bounded conductor such that $\chi\equiv \chibar\pmod*{\pfrak}$ and $\chi'\equiv \chibar'\pmod*{\pfrak}$.

Now consider the congruence $\tr{\rho_{A,\pfrak}}\equiv \chi + \chi'\pmod*{\pfrak}$. Choose an explicit $\ell$ prime to $N$ and a prime $\lambda\vert (\ell)$ of $\o_E$ above $\ell$ and evaluate said congruence on an explicit Faltings-Serre set of primes with respect to $\lambda$. Both sides are algebraic integers of explicitly bounded height, so congruence modulo the huge $\pfrak$ implies equality --- so in fact $\tr{\rho_{A,\lambda}} = \chi + \chi'$ on the Faltings-Serre set (note that we have used compatibility of Frobenius traces at $\pfrak$ and at $\lambda$), whence $\rho_{A,\lambda}$ is reducible, a contradiction.

So indeed $\rhobar_{A,\pfrak}$ is irreducible for $p$ explicitly sufficiently large, completing our sketch.
\end{proof}

\subsection{Faltings' Lemma.}

The second is Lemma $3.1$ of \cite{my-first-effective-mordell-paper}.

\begin{lem}[Faltings --- see Lemma $3.1$ of \cite{my-first-effective-mordell-paper}.]\label{faltings' lemma}
Let $d\in \Z^+$. Let $N\in \Z^+$. Let $X\in \Z^+$. Let $K/\Q$ be a number field. Then: there is an explicit (thus effectively computable) constant $C_{d,K,N,X}\in \Z^+$ such that the following holds.
\begin{itemize}
\item {Let $\o$ be an order in the ring of integers of a number field. Let $\lambda\vert (\ell)$ be a prime of $\o$ such that $\#|\o/\lambda|\leq X$. Let $\rho, \rho': \Gal(\Qbar/K)\to \GL_d(\o_\lambda)$ be unramified outside primes dividing $N\ell$ and such that $\tr(\rho(\Frob_\pfrak)) = \tr(\rho'(\Frob_\pfrak))$ for all primes $\pfrak\subseteq \o_K$ of $\o_K$ with $\pfrak\nmid (N\ell)$ and $\Nm\,{\pfrak}\leq C_{d,K,N,X}$.

Then: $\tr\circ \rho = \tr\circ \rho'$ on $\Gal(\Qbar/K)$, and $$\Z_p[\rho(\Gal(\Qbar/K))] = \Z_p[\{\rho(\Frob_\pfrak) : \pfrak\nmid (N\ell), \Nm\,{\pfrak}\leq C_{d,K,N,X}\}].$$}
\end{itemize}
\end{lem}\noindent
(The second statement follows from Nakayama exactly as in the proof of Lemma $3.1$ of \cite{my-first-effective-mordell-paper}.)

\subsection{An observation of Silverberg.}

The third is Theorem $3.5$ of \cite{my-first-effective-mordell-paper}.

\begin{lem}[Silverberg, Grothendieck --- see Theorem $3.5$ of \cite{my-first-effective-mordell-paper}.]\label{everything happens over an explicit finite extension}
Let $g\in \Z^+$. Let $K/\Q$ be a number field. Let $S$ be a finite set of places of $K$. Then: there is an explicit finite Galois extension $K'/K$ depending only on $g$, $K$, and $S$ such that the following holds.
\begin{itemize}
\item Let $A/K$ be a $g$-dimensional abelian variety over $K$ with good reduction outside $S$. Then: $A/K'$ is split semistable and $\End_{K'}(A) = \End_{\Qbar}(A)$.

Consequently its $K'$-isogeny decomposition $A\sim_{K'} \prod_i B_i^{\times n_i}$ into $K'$-simple pairwise non-$K'$-isogenous $B_i/K'$ is such that all $B_i/\Qbar$ are $\Qbar$-simple and pairwise non-$\Qbar$-isogenous.
\end{itemize}
\end{lem}

\subsection{The isogeny factorization of a $\GL_2$-type abelian variety.}

The fourth is Lemma $3.7$ of \cite{my-first-effective-mordell-paper}.

\begin{lem}[Lemma $3.7$ of \cite{my-first-effective-mordell-paper}.]\label{GL2-type abelian varieties are isotypic}
Let $K/\Q$ be a number field. Let $A/K$ be an abelian variety which is of $\GL_2$-type over $K$. Then: either there is a $K$-simple abelian variety $B/K$ of $\GL_2$-type over $K$ such that $A\sim_K B^{\times \frac{\dim{A}}{\dim{B}}}$, or else there are two $K$-simple abelian varieties $B_1, B_2/K$ which are non-$K$-isogenous and admit sufficiently many complex multiplications over $K$ such that $A\sim_K B_1^{\times \frac{\dim{A}}{2\dim{B_1}}}\times B_2^{\times \frac{\dim{A}}{2\dim{B_2}}}$.
\end{lem}

\section{A priori control of the endomorphism algebra.\label{a priori control on the endomorphism algebra subsection}}

\begin{prop}\label{the endomorphism algebra is one of an explicit finite set of possibilities}
Let $g\in \Z^+$. Let $K/\Q$ be a number field. Let $S$ be a finite set of places of $K$. Then: there is an explicit finite set of (isomorphism classes of) $\Q$-algebras $\mathcal{R}_{g,K,S}^0$ depending only on $g$, $K$, and $S$ such that the following holds.
\begin{itemize}
\item Let $A/K$ be a $g$-dimensional abelian variety of $\GL_2$-type over $K$ which has good reduction outside $S$. Then: $\End_K^0(A)\in \mathcal{R}_{g,K,S}^0$.
\end{itemize}
\end{prop}

\begin{proof}
Let $N := \prod_{\pfrak\in S} \Nm\,{\pfrak}$. Let $\ell$ be a prime of $\Z$ which is prime to $N$. Let $X := \ell^{10^{10}\cdot g^{10^{10}}}$. Let $C_{2,K,N,X}$ be the explicit constant produced by the proof of Theorem \ref{irreducibility after some point}. Let $T$ be the finite set of primes $\qfrak\subseteq \o_K$ of $\o_K$ such that $\qfrak\nmid (N\ell)$ and $\Nm\,{\qfrak}\leq C_{2,K,N,X}$.

Let $A/K$ be an abelian variety of $\GL_2$-type over $K$. By Lemma \ref{GL2-type abelian varieties are isotypic} we see that either $A\sim_K B^{\times \frac{\dim{A}}{\dim{B}}}$ with $B/K$ $K$-simple and of $\GL_2$-type over $K$, or else $A\sim_K B_1^{\times \frac{\dim{A}}{2\dim{B_1}}}\times B_2^{\times \frac{\dim{A}}{2\dim{B_2}}}$ with $B_1,B_2/K$ $K$-simple and admitting sufficiently many complex multiplications over $K$. Note that the $K$-isogeny factors of $A/K$ have good reduction outside $S$ (for example by the N\'{e}ron-Ogg-Shafarevich criterion aka by consideration of N\'{e}ron models).

Let us deal with the second case, which we will call the "CM case". That is, let us deal now with the case that $A\sim_K B_1^{\times \frac{\dim{A}}{2\dim{B_1}}}\times B_2^{\times \frac{\dim{A}}{2\dim{B_2}}}$ with $B_1,B_2/K$ $K$-simple, with good reduction outside $S$, and admitting sufficiently many complex multiplications over $K$. Then the $B_i/K$ are in fact explicitly determined: writing $K^{\text{CM}}\subseteq K$ for the maximal CM subfield of $K$, we see that for each $i$ there is a CM type $\Phi_i\subseteq \Hom_{\Q\text{-alg.}}(K^{\CM}, \C)$ such that, writing $\Phi_i'\subseteq \Hom_{\Q\text{-alg.}}(K_i, \C)$ for the corresponding reflex CM type of the reflex field $K_i\subseteq \C$, the CM Hecke character corresponding to the $\ell$-adic Galois representations $\Gal(\Qbar/K)\to \Qbar_\ell^\times$ associated to $B_i/K$ is the restriction of the CM character associated to $\Phi_i$ (which is defined over a subfield of $K^{\CM}$, since the double reflex CM type $\Phi_i''$ is the primitive CM type which induces $\Phi_i$). Thus $\End_K^0(B_i)\iso K_i$, and so $\End_K^0(A)\iso M_{\frac{\dim{A}}{2\dim{B_1}}}(K_1)\times M_{\frac{\dim{A}}{2\dim{B_2}}}(K_2)$. Since the $K_i$ are determined by the corresponding CM types $\Phi_i\subseteq \Hom_{\Q\text{-alg.}}(K^{\CM},\C)$, we conclude that in this case $\End_K^0(A)$ is determined up to an explicit finite set of possibilities depending only on $g$, $K$, and $S$. The CM case follows.

Now we move to the isotypic case, i.e.\ we consider those $g$-dimensional $A/K$ with good reduction outside $S$ and of $\GL_2$-type over $K$ for which there is a $K$-simple abelian variety $B/K$ which is of $\GL_2$-type over $K$ and such that $A\sim_K B^{\times \frac{\dim{A}}{\dim{B}}}$. We immediately dispose of the case that $B/K$ admits sufficiently many complex multiplications over $K$ by taking $B_1 := B$ and $B_2 := 0$ in the above argument treating the CM case. Write $n := \frac{\dim{A}}{\dim{B}}$ (thus $\dim{B} = \frac{g}{n}$), and $D := \End_K^0(B)$. Let $F$ be the centre of $D$ --- by the Albert classification $F/\Q$ is CM. Let $d$ be the index of the division algebra $D$ over $F$. Because $B/K$ does not admit sufficiently many complex multiplications over $K$, by the Albert classification we find that either $d = 1$ or $F/\Q$ is totally real.

Let $E/\Q$ with $E\inj M_n(D)\simeq \End_K^0(A)$ be CM and such that $[E:\Q] = g$. By tensoring the inclusion $E\inj M_n(D)$ up to $\C$ we deduce that $E\otimes_\Q \C\iso \C^g\inj M_n(D\otimes_\Q \C)\iso M_{dn}(\C)^{\oplus [F:\Q]}$. Therefore $g\leq dn\cdot [F:\Q]$. By the Albert classification we have that either $F/\Q$ is totally real and $d\leq 2$ with $d\cdot [F:\Q]\,\big\vert\, \dim{B} = \frac{g}{n}$, or else that $F/\Q$ is imaginary CM and (because we have dealt with the CM case) $d > 1$ with $\frac{d^2\cdot [F:\Q]}{2}\,\big\vert\, \dim{B} = \frac{g}{n}$. We deduce that $d\leq 2$ in both cases, and that $\frac{g}{n} = d\cdot [F:\Q]$ so that $B/K$ is of $\GL_2$-type over $K$.

Let us now show that the centre $F$ of $D$ is determined up to an explicit finite set of possibilities depending only on $g$, $K$, and $S$.\footnote{This part of the argument is an effectivization of an argument of Ribet (cf.\ in particular his proofs of Propositions $3.5$ and $5.2$ in his \cite{ribet} or his proof of Theorem $2.3$ in his \cite{ribet-endomorphisms-of-semistable-abelian-varieties}).}

Let $C\subseteq M_n(D)$ be the commutant of $E\inj M_n(D)$. Of course $F\subseteq C$. We claim that $C = E$. Indeed $C$ is a division algebra because if $\phi: B^{\times n}\to B^{\times n}$ with $0\neq \phi\in C$, then its image is $K$-isogenous to $B^{\times k}$ with $k\leq n$, and because $\phi$ commutes with $E$ it follows that $E\inj M_k(D)$, which produces (after tensoring with $\C$ as above) the inequality $g\leq kd\cdot [F:\Q] = \frac{k}{n}\cdot g$, so $k = n$ and $\phi$ is surjective and thus a $K$-isogeny. Moreover, writing $E'/E$ for the centre of $C$, we find in precisely the same way that $[E':\Q]\leq nd\cdot [F:\Q] = g = [E:\Q]$ so that $E' = E$. Thus $C$ is a division algebra with centre $E$. Writing $\delta\in \Z^+$ for the index of $C$ over $E$, we deduce from tensoring the inclusion $C\inj M_n(D)$ up to $\C$ the inclusion $M_\delta(\C)^{\oplus [E:\Q]}\inj M_{nd}(\C)^{\oplus [F:\Q]}$ and thus the inequality $\delta g\leq nd\cdot [F:\Q] = g$. Thus $\delta = 1$ and so $C = E$. In particular $F\subseteq E$.

Let now $\lambda\subseteq \o_E$ be a prime of $\o_E$ with $\lambda\vert (\ell)$. Write $\lambda'\subseteq \o_F$ with $\lambda\vert \lambda'$ for the prime of $\o_F$ below $\lambda$. Let $\rho_{A,\lambda}: \Gal(\Qbar/K)\to \GL_2(\o_{E,\lambda})$ be the $\lambda$-adic representation corresponding to $E\inj \End_K^0(A)$. Let $\rho_{B,\lambda'}: \Gal(\Qbar/K)\to (D^{\text{opp.}}\otimes_F F_{\lambda'})^\times\subseteq \GL_{2d}(F_{\lambda'})$ be the $\lambda'$-adic representation (produced via e.g.\ the double centralizer theorem) corresponding to $D\inj \End_K^0(B)$. Write, for $\qfrak\in T$, $$a_\qfrak := \tr(\rho_{A,\lambda}(\Frob_\qfrak))\in \o_E.$$ Note that $a_\qfrak = \tr(\rho_{B,\lambda'}(\Frob_\qfrak))\in \o_F$ as well.

We claim that $\Q(\{a_\qfrak\}_{\qfrak\in T}) = F$.

We have noted the forward inclusion already, so let us show the reverse inclusion via an effective form of a standard argument (cf.\ e.g.\ Ribet's proof of Proposition $3.5$ in his \cite{ribet}). For $\sigma: F\inj \Qbar_\ell$, write $\lambda_\sigma'\vert (\ell)$ for the corresponding prime of $\o_F$. Similarly, for $\tau: F'\inj \Qbar_\ell$, write $\lambda_\tau''\vert (\ell)$ for the corresponding prime of $\o_{F'}$. Note that then $\sigma$ extends to an embedding $F_{\lambda_\sigma'}\inj \Qbar_\ell$, and similarly $\tau$ extends to an embedding $F_{\lambda_\tau''}'\inj \Qbar_\ell$. From the isomorphisms
\begin{align*}
\rho_{B,\ell}\otimes_{\Q_\ell} \Qbar_\ell&\simeq \bigoplus_{\tau: F'\shortinj \Qbar_\ell} \rho_{B,\lambda_\tau''}\otimes_{F_{\lambda_\tau''}', \tau} \Qbar_\ell
\\&\simeq \bigoplus_{\sigma: F\shortinj \Qbar_\ell}\bigoplus_{\tau: F'\shortinj \Qbar_\ell \text{ s.t.\,} \tau\vert_F = \sigma} \rho_{B,\lambda_\tau''}\otimes_{F_{\lambda_\tau''}', \tau} \Qbar_\ell,
\end{align*}
and $$\End_{\Qbar_\ell[\Gal(\Qbar/K)]}(\rho_{B,\ell}\otimes_{\Q_\ell} \Qbar_\ell)\simeq \bigoplus_{\sigma: F\shortinj \Qbar_\ell} M_2(\Qbar_\ell)$$ (via Faltings' proof of the Tate conjecture for endomorphisms of abelian varieties), we conclude that, for $\tau,\tau': F'\inj \Qbar_\ell$, there is an isomorphism of absolutely irreducible representations $\rho_{B,\lambda_\tau''}\otimes_{F_{\lambda_\tau''}', \tau} \Qbar_\ell\iso \rho_{B,\lambda_{\tau'}''}\otimes_{F_{\lambda_{\tau'}''}', \tau'} \Qbar_\ell$ if and only if $\tau\vert_F = \tau'\vert_F$. In particular, if $\tau\vert_F\neq \tau'\vert_F$ the representations $\rho_{B,\lambda_\tau''}$ and $\rho_{B,\lambda_{\tau'}''}$ are not isomorphic.

Now we use Lemma \ref{faltings' lemma}. For each $\tau: F'\inj \Qbar_\ell$ with $\sigma := \tau\vert_F$, $$\sigma(a_\qfrak) = \tr(\rho_{B,\lambda_\tau''}(\Frob_\qfrak)).$$ Because the tuple $(\sigma(a_\qfrak))_{\qfrak\in T}$ determines the isomorphism class of the irreducible representation $\rho_{B,\lambda_\tau''}$ by our choice of $T$ and Lemma \ref{faltings' lemma}, we have proven that $\sigma(a_\qfrak) = \sigma'(a_\qfrak)$ for all $\qfrak\in T$ only if $\sigma = \sigma'$ as embeddings $F\inj \Qbar_\ell$. Therefore $\Q(\{a_\qfrak\}_{\qfrak\in T}) = F$ as desired.

Now because $a_\qfrak\in \o_F$ and also because by Weil's \cite{weil} for all infinite places $v$ of $F$ we have that $|a_\qfrak|_v\leq 2\sqrt{\Nm\,{\qfrak}}$, it follows that all the $a_\qfrak$ are algebraic integers of bounded height. Thus $F/\Q$ is an extension of bounded degree and discriminant and thus by Minkowski's proof of the Hermite-Minkowski theorem lies in an explicit finite set of possibilities depending only on $g$, $K$, and $S$.

So we have deduced that the centre $F$ of $D$ is determined up to an explicit finite set of possibilities depending only on $g$, $K$, and $S$. Let us next show that furthermore the absolute norm of the discriminant of $D/F$ is $\ll_{g,K,S} 1$.

The statement is of course evident when $d = 1$ (and thus $D = F$) since said discriminant is $1$. Thus we need only treat the case $d = 2$, i.e.\ the case that $D/F$ is a quaternion algebra. Let $F'/F$ be a quadratic extension splitting $D$. Let $p\geq 10^{10}$ be a prime of $\Z$. Let $\pfrak\subseteq \o_F$ be a prime of $\o_F$ such that $\pfrak\vert (p)$ and $D$ is ramified at $\pfrak$, i.e.\ $D_\pfrak := D\otimes_F F_\pfrak$ is a division algebra over $F_\pfrak$. Because $F'$ splits $D$ it follows that $\pfrak$ is not split in $F'$ (else $F'_{\pfrak'}\simeq F_\pfrak$ for all primes $\pfrak'\subseteq \o_{F'}$ of $\o_{F'}$ with $\pfrak'\vert \pfrak$, and so $D\otimes_F F_{\pfrak'}'$ would be a division algebra). Let $\pfrak'\subseteq \o_{F'}$ with $\pfrak'\vert \pfrak$ be the unique prime of $\o_{F'}$ above $\pfrak$. Let $\pi$ be a uniformizer of $\pfrak$. Let $u\in \o_{F,\pfrak}^\times$ be a nonsquare unit. Recall from the classification of quaternion algebras over $p$-adic fields that there is a unique quaternion algebra, namely the one with symbol $(\pi, u)_\pfrak$, which is a division algebra over $F_\pfrak$, and so $D_\pfrak$ is the quaternion algebra over $F_\pfrak$ with symbol $(\pi, u)_\pfrak$. Thus $D_\pfrak$ has $F_\pfrak$-basis $(1, i, j, k)$ with $i^2 = \pi$, $j^2 = u$, and $k := ij = -ji$. Moreover it has a unique maximal order, namely $\o_{D,\pfrak} := \o_{F,\pfrak}\cdot 1 + \o_{F,\pfrak}\cdot i + \o_{F,\pfrak}\cdot j + \o_{F,\pfrak}\cdot k$ (aka the set of elements with integral norm).

Now because $\pfrak$ is not split in $F'$ (and $p$ is large) it follows that either $F_{\pfrak'}' = F_\pfrak(\sqrt{\pi})$ or else $F_{\pfrak'}' = F_\pfrak(\sqrt{u})$. We break into cases.

First let us deal with the ramified case, i.e.\ the case $F_{\pfrak'}' = F_\pfrak(\sqrt{\pi})$. Let $\pi' := \sqrt{\pi}$ be a uniformizer of $\pfrak'$. Recall that $\o_{F',\pfrak'} = \o_{F,\pfrak}[\pi']$ and that $\o_F/\pfrak\simeq \o_{F'}/\pfrak'$. Because $D_\pfrak$ is a division algebra and $F'\inj D$ (and so $F_{\pfrak'}'\inj D_\pfrak$), it follows\footnote{Writing $\alpha\in D$ for the image of $\sqrt{\pi}\in F_{\pfrak'}'$ and $\ell_i, r_\alpha: D\to D$ for the (commuting) $F$-linear transformations given by $\ell_i(x) = i\cdot x$ and $r_\alpha(x) = x\cdot \alpha$, we have that $(\ell_i - r_\alpha)\cdot (\ell_i + r_\alpha) = 0$, whence one of the two factors has a nonzero kernel, and a nonzero element in said kernel conjugates $\alpha$ to $\pm i$.} that we may without loss of generality take the embedding $F_{\pfrak'}'\inj D_\pfrak$ to be via $\pi'\mapsto \pm i$. By precomposing with an element of $\Gal(F_{\pfrak'}'/F_\pfrak) = \Gal(F'/F)$ if necessary we may and will arrange that the embedding is via $\pi'\mapsto i$. Note then that $\o_{D,\pfrak} = \o_{F',\pfrak'} + \o_{F',\pfrak'}\cdot j$.

Using this basis (i.e.\ $(1,j)$), we see that the action of $\o_{D,\pfrak}$ on $\o_{D,\pfrak}$, regarded as a free $\o_{F',\pfrak'}$-module of rank $2$, is given by $\o_{D,\pfrak}\to \GL_2(\o_{F',\pfrak'})$ via $$a + bi + cj + dk\mapsto \twobytwo{a + b\pi'}{u\cdot (c + d\pi')}{c + d\pi'}{a - b\pi'}.$$ Therefore the action of $\o_{D,\pfrak}\otimes_{\o_{F,\pfrak}} \o_F/\pfrak$ on $\o_{D,\pfrak}\otimes_{\o_{F,\pfrak}} \o_F/\pfrak = \o_{D,\pfrak}\otimes_{\o_{F',\pfrak'}} \o_{F'}/\pfrak'$ is given by $$a + bi + cj + dk\mapsto \twobytwo{a}{u\cdot c}{c}{a}\bmod{\pfrak'}\in \GL_2(\o_F/\pfrak),$$ which is visibly abelian (indeed the image lies in the nonsplit Cartan) and thus not absolutely irreducible.

Therefore it follows that the $\pfrak'$-adic representation $\rho_{B,\pfrak'}: \Gal(\Qbar/K)\to (\o_{D,\pfrak}^\opp)^\times\inj \GL_2(\o_{F',\pfrak'})$ does not have an absolutely irreducible residual representation. We conclude from Theorem \ref{irreducibility after some point} that $p\ll_{g,K,S} 1$.

We treat the unramified case, i.e.\ the case $F_{\pfrak'}' = F_\pfrak(\sqrt{u})$, in the same way. Now we choose as uniformizer $\pi' := \pi$, observe that $\o_{F',\pfrak'} = \o_{F,\pfrak}[\sqrt{u}]$, and see that, up to precomposing with an element of $\Gal(F_{\pfrak'}'/F_\pfrak) = \Gal(F'/F)$ if necessary, we may without loss of generality take the embedding $F_{\pfrak'}'\inj D_\pfrak$ to be given by $\sqrt{u}\mapsto j$. Thus we may use $(1, i)$ as our $\o_{F',\pfrak'}$-basis of $\o_{D,\pfrak}$, and in this basis the action of $\o_{D,\pfrak}$ on $\o_{D,\pfrak}$, regarded as a free $\o_{F',\pfrak'}$-module of rank $2$, is given by $$a + bi + cj + dk\mapsto \twobytwo{a + c\sqrt{u}}{\pi\cdot (b + d\sqrt{u})}{b + d\sqrt{u}}{a - c\sqrt{u}}.$$ Thus the action of $\o_{D,\pfrak}\otimes_{\o_{F,\pfrak}} \o_F/\pfrak$ on $\o_\pfrak\otimes_{\o_{F',\pfrak'}} \o_{F'}/\pfrak'$ is given by $$a + bi + cj + dk\mapsto \twobytwo{a + c\sqrt{u}}{0}{b + d\sqrt{u}}{a - c\sqrt{u}}\bmod{\pfrak'}\in \GL_2(\o_{F'}/\pfrak'),$$ which is again visibly not irreducible since its image lies in a Borel.

Therefore it follows that the $\pfrak'$-adic representation $\rho_{B,\pfrak'}: \Gal(\Qbar/K)\to (\o_{D,\pfrak}^\opp)^\times\inj \GL_2(\o_{F',\pfrak'})$ does not have absolutely irreducible residual representation, and so we again conclude that $p\ll_{g,K,S} 1$.

So in sum we have bounded the absolute norm of the discriminant of $D/F$, since we have shown that the only primes at which $D/F$ may ramify are of norm $\ll_{g,K,S} 1$.

We conclude that there is an explicit finite set of possibilities depending only on $g$, $K$, and $S$ for $D = \End_K^0(B)$. We are done.
\end{proof}

Let us also record the following corollary.

\begin{cor}\label{the geometric endomorphism algebra is one of an explicit finite set of possibilities}
Let $g\in \Z^+$. Let $K/\Q$ be a number field. Let $S$ be a finite set of places of $K$. Then: there is an explicit finite set of (isomorphism classes of) $\Q$-algebras $\widetilde{\mathcal{R}}_{g,K,S}^0$ depending only on $g$, $K$, and $S$ such that the following holds.
\begin{itemize}
\item Let $A/K$ be a $g$-dimensional abelian variety of $\GL_2$-type over $K$ which has good reduction outside $S$. Then: $\End_{\Qbar}^0(A)\in \widetilde{\mathcal{R}}_{g,K,S}^0$.
\end{itemize}
\end{cor}

\begin{proof}
Let $K'/K$ be the explicit finite Galois extension produced by the proof of Lemma \ref{everything happens over an explicit finite extension} and let $\widetilde{\mathcal{R}}_{g,K,S}^0 := \mathcal{R}_{g,K',S}^0$.
\end{proof}

\section{Two "large image" propositions.}

Now let us state the two intermediate results we will use.

\subsection{Large primes.}

Let us first deal with large primes.

\begin{prop}\label{large prime lemma}
Let $g\in \Z^+$. Let $K/\Q$ be a number field. Let $S$ be a finite set of places of $K$. Then: there is an explicit (thus effectively computable) constant $C_{g,K,S}\in \Z^+$ depending only on $g$, $K$, and $S$ such that the following two statements hold.
\begin{itemize}
\item Let $p\geq C_{g,K,S}$ be a prime of $\Z$. Let $F/\Q$ be a number field of degree $[F:\Q] = g$. Let $A/K$ be a $g$-dimensional abelian variety over $K$ admitting $\o_F\simeq \End_K(A) = \End_{\Qbar}(A)$ which has good reduction outside $S$ and does not admit sufficiently many complex multiplications over $\Qbar$. Write $\rho_{A,p}: \Gal(\Qbar/K)\to \GL_2(\o_{F,p})$ for the $2$-dimensional representation corresponding to the $\o_F[\Gal(\Qbar/K)]$-module structure of $T_p(A)$. Then: $\Z_p[\rho_{A,p}(\Gal(\Qbar/K))] = M_2(\o_{F,p})$.
\item Let $p\geq C_{g,K,S}$ be a prime of $\Z$. Let $F/\Q$ be a number field of degree $[F:\Q] = \frac{g}{2}$ and let $D/F$ be a quaternion algebra over $F$. Let $\o\subseteq D$ be a maximal order. Let $A/K$ be a $g$-dimensional abelian variety over $K$ admitting $\o\simeq \End_K(A) = \End_{\Qbar}(K)$ which has good reduction outside $S$ and does not admit sufficiently many complex multiplications over $\Qbar$. Write $\rho_{A,p}: \Gal(\Qbar/K)\to (\o_p^{\opp})^\times$ for the representation corresponding to the $\o[\Gal(\Qbar/K)]$-module structure of $T_p(A)$. Then: $\Z_p[\rho_{A,p}(\Gal(\Qbar/K))] = \o_p^\opp$.
\end{itemize}
\end{prop}

\begin{proof}
Certainly in both cases $A/K$ is $K$-simple.

By Proposition \ref{the endomorphism algebra is one of an explicit finite set of possibilities} it follows that $\End_K^0(A)$ lies in an explicit finite set. By the finiteness of the class number, whence type number, of a quaternion algebra over a number field, it follows that the maximal order $\End_K(A)$ lies in an explicit finite set as well.

Let $N := \prod_{\pfrak\in S} (\Nm\,{\pfrak})$. Let $\ell\in \Z^+$ be the smallest prime of $\Z$ which is prime to $10^{10}!\cdot N$.

Let $T$ be the finite set of primes produced by Lemma \ref{faltings' lemma} with parameters $(2,K,N,\ell^{10^{10 g}})$.

Let, via Minkowski's proof of the Hermite-Minkowski theorem, $L/K$ be an explicit finite extension such that, for all finite-order characters $\psi: \Gal(\Qbar/K)\to \C^\times$ of conductor dividing $\Delta_N := N^{10^{10 g}\cdot [K:\Q]}$, $\psi\vert_{\Gal(\Qbar/L)} = \triv$.

Let us deal with the first case first. Let $A/K$ be a $g$-dimensional abelian variety over $K$ admitting $\o_F\simeq \End_K(A)$ which has good reduction outside $S$ and which does not admit sufficiently many complex multiplications over $\Qbar$.

Now $\chi_p^{-1}\cdot \det{\rho_{A,p}}$ is finite-order and of conductor dividing $\Delta_N$. Thus $\det{\rho_{A,p}}\vert_{\Gal(\Qbar/L)} = \chi_p$ and in particular has image in $\Z_p^\times$.

It of course suffices to show the claim after restricting to $\Gal(\Qbar/L)$, so without loss of generality $L = K$ and $\det{\rho_{A,p}}$ is valued in $\Z_p^\times$.

Now for all $g\in \GL_2(\o_{F,p})$ we have that $$g + (\det{g})\cdot g^{-1} = (\tr\,{g})\cdot \id.$$ Applying this identity to all $\Frob_\qfrak$ with $\qfrak\in T$ we find that $$\Z_p[\rho_{A,p}(\Gal(\Qbar/K))]\supseteq \Z[\{\tr(\rho_{A,p}(\Frob_\qfrak)) : \qfrak\in T\}]\otimes_\Z \Z_p.$$

Of course by strict compatibility we have that $\tr(\rho_{A,p}(\Frob_\qfrak))\in \o_F\inj \o_{F,p}$ for all $\qfrak\in T$.

Now by the second part of Lemma \ref{faltings' lemma} it follows that $\Q_\ell[\tr(\rho_{A,p}(\Gal(\Qbar/K)))] = \Q_\ell[\{\tr(\rho_{A,p}(\Frob_\qfrak)) : \qfrak\in T\}]$. Writing $\rho_{A,\ell}\otimes_{\Z_\ell} \Qbar_\ell =: \bigoplus_{\lambda\vert (\ell)} \bigoplus_{\sigma: F_\lambda\inj \Qbar_\ell} \rho_{A,\lambda}\otimes_{\o_{F,\lambda}, \sigma} \Qbar_\ell$ (with the obvious notation), it follows from Faltings' proof of the Tate conjecture that the $\rho_{A,\lambda}\otimes_{\o_{F,\lambda}, \sigma} \Qbar_\ell$ are pairwise non-isomorphic as $\Qbar_\ell$-representations.

By Lemma \ref{faltings' lemma} and Brauer-Nesbitt it follows that, for all embeddings $\sigma, \sigma': F\inj \Qbar_\ell$, the equality $\sigma(\tr(\rho_{A,\ell}(\Frob_\qfrak))) = \sigma'(\tr(\rho_{A,\ell}(\Frob_\qfrak)))$ holds for all $\qfrak\in T$ if and only if $\sigma = \sigma'$.

Consequently $\Z[\{\tr(\rho_{A,\ell}(\Frob_\qfrak) : \qfrak\in T\}]\subseteq \o_F$ is of finite index. Since by e.g.\ purity the generators $\tr(\rho_{A,\ell}(\Frob_\qfrak))$ are of bounded height, it follows that the discriminant of $\Z[\{\tr(\rho_{A,\ell}(\Frob_\qfrak) : \qfrak\in T\}]$ is $\ll_{g,K,S} 1$, whence the same bound holds for its index in $\o_F$, whence we have that $$\Z_p[\{\tr(\rho_{A,\ell}(\Frob_\qfrak) : \qfrak\in T\}] = \o_{F,p}.$$

We conclude that $\Z_p[\rho_{A,p}(\Gal(\Qbar/K))]\supseteq \o_{F,p}$, or in other words that $$\Z_p[\rho_{A,p}(\Gal(\Qbar/K))] = \o_{F,p}[\rho_{A,p}(\Gal(\Qbar/K))] = \bigoplus_{\pfrak\vert (p)} \o_{F,\pfrak}[\rho_{A,\pfrak}(\Gal(\Qbar/K))].$$

Now by Lemma \ref{irreducibility after some point} it follows that $\overline{\rho}_{A,\pfrak}(\Gal(\Qbar/K))$ contains a conjugate of $\SL_2(\F_p)$. Thus evidently $\o_{F,\pfrak}[\overline{\rho}_{A,\pfrak}(\Gal(\Qbar/K))]\otimes_{\o_F} \o_F/\pfrak = M_2(\o_F/\pfrak)$.

Thus by Nakayama we find that $\o_{F,\pfrak}[\rho_{A,\pfrak}(\Gal(\Qbar/K))] = M_2(\o_{F,\pfrak})$, whence $$\Z_p[\rho_{A,p}(\Gal(\Qbar/K))] = M_2(\o_{F,p})$$ as desired.

So we have dealt with the first case.

But the above treats the second case as well since $D_p\simeq M_2(F_p)$ because $p$ is large (since by Proposition \ref{the endomorphism algebra is one of an explicit finite set of possibilities} the discriminant of $\o$ is $\ll_{g,K,S} 1$). By uniqueness up to conjugation of the maximal order $M_2(\o_{F,\pfrak})$ of $M_2(F_\pfrak)$, it follows that, up to conjugation by an element of $D_p^\times$, $\o_p\simeq M_2(\o_{F,p})$, and now the entire argument goes through verbatim. So we are done.
\end{proof}

\subsection{Small primes.}

Now let us prove the corresponding statement for small primes.

\begin{prop}\label{small prime lemma}
Let $g\in \Z^+$. Let $K/\Q$ be a number field. Let $S$ be a finite set of places of $K$. Let $p$ be a prime of $\Z$. Then: there is an effectively computable constant $C_{g,K,S,p}\in \Z^+$ depending only on $g$, $K$, $S$, and $p$ such that the following two statements hold.
\begin{itemize}
\item Let $F/\Q$ be a number field of degree $[F:\Q] = g$. Let $A/K$ be a split semistable $g$-dimensional abelian variety over $K$ admitting $\o_F\simeq \End_K(A) = \End_{\Qbar}(A)$ which has good reduction outside $S$ and does not admit sufficiently many complex multiplications over $\Qbar$. Write $\rho_{A,p}: \Gal(\Qbar/K)\to \GL_2(\o_{F,p})$ for the $2$-dimensional representation corresponding to the $\o_F[\Gal(\Qbar/K)]$-module structure of $T_p(A)$. Then: $$[M_2(\o_{F,p}) : \Z_p[\rho_{A,p}(\Gal(\Qbar/K))]]\leq C_{g,K,S,p}.$$
\item Let $F/\Q$ be a number field of degree $[F:\Q] = \frac{g}{2}$ and let $D/F$ be a quaternion algebra over $F$. Let $\o\subseteq D$ be a maximal order. Let $A/K$ be a split semistable $g$-dimensional abelian variety over $K$ admitting $\o\simeq \End_K(A) = \End_{\Qbar}(K)$ which has good reduction outside $S$ and does not admit sufficiently many complex multiplications over $\Qbar$. Write $\rho_{A,p}: \Gal(\Qbar/K)\to (\o_p^{\opp})^\times$ for the representation corresponding to the $\o[\Gal(\Qbar/K)]$-module structure of $T_p(A)$. Then: $$[\o_p^\opp : \Z_p[\rho_{A,p}(\Gal(\Qbar/K))]]\leq C_{g,K,S,p}.$$
\end{itemize}
\end{prop}

Of course for $p\gg_{g,K,S} 1$ Proposition \ref{small prime lemma} follows from Proposition \ref{large prime lemma}.

The proposition will follow from the following.

\begin{lem}\label{no mod-p^n characters for n large}
Let $g\in \Z^+$. Let $K/\Q$ be a number field. Let $S$ be a finite set of places of $K$. Let $p$ be a prime of $\Z$. Let $X\in \Z^+$. Then: there is an effectively computable constant $C_{g,K,S,X}\in \Z^+$ depending only on $g$, $K$, $S$, and $X$ such that the following statement holds.

\begin{itemize}
\item Let $F/\Q$ be a number field of degree $[F:\Q]\leq g$. Let $F'/F$ be an extension of degree $[F':F]\leq 2$ with relative discriminant $|\Nm\,{\Delta_{F'/F}}|\leq X$. Let $\pfrak\subseteq \o_F$ with $\pfrak\vert (p)$ be a prime of $\o_F$. Let $\pfrak'\subseteq \o_{F'}$ with $\pfrak'\vert \pfrak$ be a prime of $\o_{F'}$. Let $\chibar, \chibar': \Gal(\Qbar/K)\to \left(\o_{F'}/\pfrak'^{2n\cdot e(\pfrak'/\pfrak)}\right)^\times$. Let $A/K$ be a split semistable $g$-dimensional abelian variety over $K$ admitting $F\inj \End_K^0(A) = \End_{\Qbar}^0(A)\in \widetilde{\mathcal{R}}_{g,K,S}^0$ which is $K$-simple, has good reduction outside $S$, does not admit sufficiently many complex multiplications over $\Qbar$, and is such that $0\to \chibar\to A[\pfrak^{2n}]\otimes_{\o_{F,\pfrak}} \o_{F',\pfrak'}\to \chibar'\to 0$.

Then: $n\leq C_{g,K,S,X}$.
\end{itemize}
\end{lem}

Again for $p\gg_{g,K,S} 1$ this follows from Proposition \ref{large prime lemma}. 

To prove Lemma \ref{no mod-p^n characters for n large} we will essentially repeat verbatim our implicit proof of Lemma \ref{irreducibility after some point} via our citation of our \cite{my-thesis, my-first-effective-mordell-paper} --- we will lift the semisimplification to a direct sum of algebraic Hecke characters of bounded conductor (and weight) and then compare Frobenius traces at a Faltings-Serre set at a suitable reference prime (in this case just $\pfrak$), where a congruence mod $\pfrak^n$ for $n\gg_{g,K,S} 1$ would imply equality and thus reducibility of the $\pfrak$-adic representation, a contradiction.

However to repeat said proof we need to be able to lift a mod-$\pfrak^n$ character to an algebraic Hecke character when $n\gg_{g,K,S} 1$. That we can do this will be guaranteed by the following.

\begin{lem}\label{CM lifting}
Let $g\in \Z^+$. Let $K/\Q$ be a number field. Let $S$ be a finite set of places of $K$. Let $p$ be a prime of $\Z$. Let $X\in \Z^+$. Then: there is an effectively computable constant $C_{g,K,S,X}\in \Z^+$ depending only on $g$, $K$, $S$, and $X$ such that the following statement holds.

\begin{itemize}
\item Let $n\geq C_{g,K,S,X}$. Let $F/\Q$ be a number field of degree $[F:\Q]\leq g$. Let $F'/F$ be an extension of degree $[F':F]\leq 2$ with relative discriminant $|\Nm\,{\Delta_{F'/F}}|\leq X$. Let $\pfrak\subseteq \o_F$ with $\pfrak\vert (p)$ be a prime of $\o_F$. Let $\pfrak'\subseteq \o_{F'}$ with $\pfrak'\vert \pfrak$ be a prime of $\o_{F'}$. Let $\chibar: \Gal(\Qbar/K)\to \left(\o_{F'}/\pfrak'^{2n\cdot e(\pfrak'/\pfrak)}\right)^\times$. Let $A/K$ be a split semistable $g$-dimensional abelian variety over $K$ admitting $F\inj \End_K^0(A) = \End_{\Qbar}^0(A)\in \widetilde{\mathcal{R}}_{g,K,S}^0$ which is $K$-simple, has good reduction outside $S$, does not admit sufficiently many complex multiplications over $\Qbar$, and is such that $\chibar\inj A[\pfrak^{2n}]\otimes_{\o_{F,\pfrak}} \o_{F',\pfrak'}$.

Then: there is an algebraic Hecke character $\chi$ of $K$ of weight $\leq 1$ and conductor dividing $\prod_{\qfrak\in S} (\Nm\,{\qfrak})^{10^{10g}\cdot [K:\Q]}$ valued in an extension $F''/F'$ and a prime $\pfrak''\subseteq \o_{F''}$ with $\pfrak''\vert \pfrak'$ such that $\chi\equiv \chibar\pmod*{\pfrak''^{2n\cdot e(\pfrak''/\pfrak)}}$, where we have also written $\chi$ for the $\pfrak''$-adic character corresponding to said algebraic Hecke character.
\end{itemize}
\end{lem}

In the case of $p\gg_{g,K,S} 1$ this was proved by an argument involving class field theory and height bounds on global units generating $\o_K^\times$ to reduce to the standard classification of algebraic Hecke characters, all of which are still available when $p\ll_{g,K,S} 1$. What breaks down, however, is that to evaluate the restrictions of $\chibar$ to the various inertia groups above $p$ we used Raynaud's classification of finite flat group schemes which prolong over an unramified base (or alternatively Fontaine-Laffaille's similar classification).

It would be natural then to turn to the Breuil-Kisin classification (let us emphasize that we have a finite flat group scheme associated to a character mod $p^n$ with large $n$ when $p\ll_{g,K,S} 1$ --- of course the statement is false for e.g.\ $n = 1$).

But we will not need to thanks to a trick.

\subsubsection{Lemma \ref{no mod-p^n characters for n large} implies Proposition \ref{small prime lemma}.}

We will first explain how to prove Proposition \ref{small prime lemma} assuming Lemma \ref{no mod-p^n characters for n large}.

Let us give a quick sketch in the case of $F\simeq \End_K^0(A)$ before we begin the proof. We start with the $\Z_p$-algebra generated by the image of $\Gal(\Qbar/K)$ acting on $T_p(A)$. We will show it is sufficiently close to an $\o_{F,p}$-algebra by considering Frobenius traces (which are in said $\Z_p$-algebra by the same trick we used in Proposition \ref{large prime lemma}). This lets us deal with the representations $T_\pfrak(A)$ one by one.

We have shown, so to speak, that our algebra basically contains $\o_{F,\pfrak}\cdot \id$. To get another "dimension", we consider the actual images (not just traces) of Frobenius elements, again at a Faltings-Serre set --- said images generate the algebra by Nakayama, so there must certainly be one which does not act by a scalar ($A/K$ is not CM). Again the $\o_{F,\pfrak}$-algebra generated by that Frobenius image is basically the maximal order in the $F_\pfrak$-algebra generated by the Frobenius image (since its characteristic polynomial has bounded discriminant).

So we have gotten our algebra to contain a second dimension --- a Cartan, which for simplicity let us say is the split Cartan. Finally we apply Lemma \ref{no mod-p^n characters for n large} to show that our algebra does not reduce to a Cartan (or even into a Borel) mod $\pfrak^n$ for $n$ large --- thus in this simplified picture it contains a third element with top right corner of bounded valuation, and applying the lemma again without loss of generality said element also has bottom left corner of bounded valuation. Together the elements we have produced span a subalgebra of $M_2(\o_{F,\pfrak})$ of bounded index, as desired.

\begin{proof}[Proof of Proposition \ref{small prime lemma} assuming Lemma \ref{no mod-p^n characters for n large}.]
Again for $p\gg_{g,K,S} 1$ there is nothing to do thanks to Proposition \ref{large prime lemma}. Therefore without loss of generality $p\ll_{g,K,S} 1$.

We repeat verbatim the beginning of the proof of Proposition \ref{large prime lemma}. In the same notation, we conclude that $\Z_p[\rho_{A,p}(\Gal(\Qbar/K))]\supseteq \Z[\{\tr(\Frob_\qfrak) : \qfrak\in T\}]\otimes_\Z \Z_p$. Again we observe that $\Z[\{\tr(\Frob_\qfrak) : \qfrak\in T\}]\subseteq \o_F$ is an order of index $\ll_{g,K,S} 1$. It therefore follows that the index of $\Z_p[\rho_{A,p}(\Gal(\Qbar/K))]$ inside $\o_{F,p}[\rho_{A,p}(\Gal(\Qbar/K))] = \bigoplus_{\pfrak\vert (p)} \o_{F,\pfrak}[\rho_{A,\pfrak}(\Gal(\Qbar/K))]$ is $\ll_{g,K,S} 1$.

So it suffices to show that, for $\pfrak\subseteq \o_F$ a prime of $\o_F$ with $\pfrak\vert (p)$, the index of $\o_{F,\pfrak}[\rho_{A,\pfrak}(\Gal(\Qbar/K))]$ inside $M_2(\o_{F,\pfrak})$ is $\ll_{g,K,S} 1$.

Now, by the second part of Lemma \ref{faltings' lemma} (aka Nakayama), we have that $\o_{F,\pfrak}[\{\rho_{A,\pfrak}(\Frob_\qfrak) : \qfrak\in T\}] = \o_{F,\pfrak}[\rho_{A,\pfrak}(\Gal(\Qbar/K))]$.

Let us note before splitting into cases that we may and will expand $T$ so that for each $D\in \widetilde{\mathcal{R}}_{g,K,S}^0$ there is a $\qfrak\in T$ which is of degree $1$ in $K$ and which lies over an unramified prime of the centre of $D$.

Now we bifurcate. Let us first treat the case where $F\simeq \End_K^0(A) = \End_{\Qbar}^0(K)$.

Let $\mfrak\in T$ be such that $\Nm\,{\mfrak}\in \Z$ is prime and unramified in $F$. Let $f_\mfrak(t)\in \o_F[t]$ be the characteristic polynomial of $\rho_{A,\pfrak}(\Frob_\mfrak)$. Since $\det{(\rho_{A,\pfrak}(\mfrak))} = \Nm\,{\mfrak}$ (since our quoting of our argument beginning the proof of Proposition \ref{large prime lemma} included arranging that $\rho_{A,\pfrak}$ has cyclotomic determinant), it follows that the monic quadratic $f_\mfrak(t)$ has nonzero discriminant (else $\Nm\,{\mfrak}$ would be a square in $F$ and thus a prime of ramification of $F$).

Moreover the absolute norm of said discriminant is of course $\ll_{g,K,S} 1$, whence the index of $\o_F[t]/(f_\mfrak)$ inside the maximal order of the quadratic \'{e}tale algebra $F[t]/(f_\mfrak)$ is also $\ll_{g,K,S} 1$. Thus similarly the index of $\o_{F,\pfrak}[\rho_{A,\pfrak}(\Frob_\mfrak)]$ inside the (monogenic) maximal order of the quadratic \'{e}tale algebra $F_\pfrak[\rho_{A,\pfrak}(\Frob_\mfrak)]\simeq F_\pfrak[t]/(f_\mfrak)$ is $\ll_{g,K,S} 1$.

Therefore there is an $F'/F$ of degree $[F':F]\leq 2$ and with relative discriminant $|\Nm\,{\Delta_{F'/F}}|\ll_{g,K,S} 1$ such that there is a prime $\pfrak'\subseteq \o_{F'}$ with $\pfrak'\vert \pfrak$ for which the quadratic \'{e}tale algebra $F'_{\pfrak'}[\rho_{A,\pfrak}(\Frob_\mfrak)]\simeq F'_{\pfrak'}[t]/(f_\mfrak)$ is split.

Moreover it suffices to show that $\o_{F',\pfrak'}[\rho_{A,\pfrak}(\Gal(\Qbar/K))]\subseteq M_2(\o_{F',\pfrak'})$ has index $\ll_{g,K,S} 1$ (we have just applied ${-}\otimes_{\o_{F,\pfrak}} \o_{F',\pfrak'}$ to our previous cokernel).

By diagonalizing the regular semisimple element $\rho_{A,\pfrak}(\Frob_\mfrak)\in M_2(\o_{F',\pfrak'})$ we find that, in said basis, $\twobytwo{\o_{F',\pfrak'}}{0}{0}{\o_{F',\pfrak'}}\subseteq \o_{F',\pfrak'}[\rho_{A,\pfrak}(\Gal(\Qbar/K))]$.

By Lemma \ref{no mod-p^n characters for n large} it follows that, in this basis, there is an $\alpha\in \o_{F',\pfrak'}[\rho_{A,\pfrak}(\Gal(\Qbar/K))]$ such that $\alpha$ is not diagonal modulo $\pfrak'^n$ for an $n\ll_{g,K,S} 1$.

Applying Lemma \ref{no mod-p^n characters for n large} again if necessary (to produce an $\alpha'$ and then replacing $\alpha$ by $\alpha + \alpha'$ if necessary), we see that without loss of generality the $\pfrak$-adic valuations of the top-left and bottom-right corners of $\alpha$ are $\ll_{g,K,S} 1$.

It follows that there is an $n\ll_{g,K,S} 1$ such that $p^n\cdot M_2(\o_{F',\pfrak'})\subseteq \o_{F',\pfrak'}[\rho_{A,\pfrak}(\Gal(\Qbar/K))]$, and we are done.

Now for the case of $\End_K^0(A) = D$ a quaternion algebra over $F$. But the above argument also works verbatim --- so long as we replace $M_2(\o_{F,\pfrak})$ and $M_2(F_\pfrak)$ by $\o_\pfrak^\opp := \o^\opp\otimes_{\o_F} \o_{F,\pfrak}$ and $D_\pfrak^\opp := D^\opp\otimes_F F_\pfrak$, respectively ---  because $F'/F$ also splits $D/F$.
\end{proof}

\subsubsection{Lemma \ref{CM lifting} implies Lemma \ref{no mod-p^n characters for n large}.}

Let us next prove Lemma \ref{no mod-p^n characters for n large} assuming Lemma \ref{CM lifting}.

\begin{proof}[Proof of Lemma \ref{no mod-p^n characters for n large} assuming Lemma \ref{CM lifting}.]
Let, via Lemma \ref{CM lifting}, $\chi$ and $\chi'$ be algebraic Hecke characters of weight $\leq 1$ of conductor dividing $\Delta_N := \prod_{\qfrak\in S} (\Nm\,{\qfrak})^{10^{10g}\cdot [K:\Q]}$ and valued in an extension $F''/F'$ with a prime $\pfrak''\vert \pfrak'$ such that $\chi\equiv \chibar\pmod*{\pfrak''^{2n\cdot e(\pfrak''/\pfrak)}}$ and $\chi'\equiv \chibar'\pmod*{\pfrak''^{2n\cdot e(\pfrak''/\pfrak)}}$, where we have also written $\chi$ and $\chi'$ for the corresponding $\pfrak''$-adic characters.

Let $T$ be the finite set of primes produced by Lemma \ref{faltings' lemma} with parameters $(2,K,\Delta_N,p^{10^{10 g}})$.

Because $\tr{(\rho_{A,\pfrak}(\Frob_\mfrak))}\equiv \chi(\mfrak) + \chi(\mfrak')\pmod*{\pfrak''^{2n\cdot e(\pfrak''/\pfrak)}}$ for all $\mfrak\in T$ and because both sides are algebraic integers of height and degree $\ll_{g,K,S} 1$, it follows that if $n\gg_{g,K,S} 1$ were explicitly sufficiently large said congruence would be an equality for all $\mfrak\in T$, whence by Lemma \ref{faltings' lemma} $\rho_{A,\pfrak}$ would have semisimplification $\chi\oplus \chi'$ and would in particular be reducible, a contradiction.
\end{proof}

\subsubsection{Preliminaries for the proof of Lemma \ref{CM lifting}.}

To prove Lemma \ref{CM lifting} we will need a few preliminary lemmas. First we note the following easy fact.

\begin{lem}\label{recovering characters from a direct sum}
Let $G$ be a group. Let $F/\Q$ be a number field. Let $\pfrak\subseteq \o_F$ be a prime of $\o_F$. Let $n\in \Z^+$. Let $\chi_1,\ldots,\chi_4: G\to (\o_F/\pfrak^{2n})^\times$ be such that $\chi_1 + \chi_2 = \chi_3 + \chi_4$ and $\chi_1\cdot \chi_2 = \chi_3\cdot \chi_4$ as functions on $G$. Then: $\{\chi_1, \chi_2\}\equiv \{\chi_3, \chi_4\}\pmod*{\pfrak^n}$.
\end{lem}

\begin{proof}
By replacing each $\chi_i$ by $\chi_i\cdot \chi_1^{-1}$ it suffices to treat the case where $\chi_1 = \triv$. Thus $(t - \chi_3(g))\cdot (t - \chi_4(g)) = (t - 1)\cdot (t - \chi_2(g))\in (\o_F/\pfrak^{2n})[t]$ for all $g\in G$, and so taking $t = 1$ we conclude that $(1 - \chi_3(g))\cdot (1 - \chi_4(g)) = 0$ for all $g\in G$, which is to say that, for all $g\in G$, either $\chi_3(g)\equiv 1\pmod*{\pfrak^n}$ or $\chi_4(g)\equiv 1\pmod*{\pfrak^n}$.

Let $\psi_i := \chi_i\pmod*{\pfrak^n}$, whence $\psi_i: G\to (\o_F/\pfrak^n)^\times$. Let $H_i := \ker{\psi_i}$. Thus we have found that $G = H_3\cup H_4$. So either $H_3 = G$, or else, letting $g\in G - H_3$ (and thus $g\in H_4$), we see that $g\cdot H_3\subseteq H_4$ and so $H_3\subseteq H_4$, whence $H_4 = G$.

So we conclude that either $\psi_3 = \triv$ or $\psi_4 = \triv$. Subtracting from $\triv + \psi_2 = \psi_3 + \psi_4$ we find that $\{\triv, \psi_2\} = \{\psi_3, \psi_4\}$, as desired.
\end{proof}

Next we will evaluate the inertial restriction of a relevant character at a prime of bad reduction of the given $A/K$.

\begin{lem}\label{description of the inertial representation of a GL2-type abelian variety with bad reduction}
Let $K/\Q$ be a number field. Let $\qfrak\subseteq \o_K$ be a prime of $K$. Let $p\in \Z^+$ be the prime of $\Z$ with $\qfrak\vert (p)$. Let $A/K_\qfrak$ be an abelian variety over $K_\qfrak$ with split semistable bad reduction at $\qfrak$ which is of $\GL_2$-type over $K_\qfrak$. Let $\o$ be an order in a CM field $F/\Q$ admitting $\o\inj \End_{K_\qfrak}(A)$ and such that $[F:\Q] = \dim{A}$. Let $\pfrak\subseteq \o$ with $\pfrak\vert (p)$ be a prime of $\o$ above $p$. Write $\rho_{A,\pfrak}: \Gal(\Qbar/K)\to \GL_2(\o_\pfrak)$ for the $2$-dimensional representation corresponding to the $\o[\Gal(\Qbar/K)]$-module structure on $T_\pfrak(A)$.

Then: $0\to \chi_p\to \rho_{A,\pfrak}\vert_{I_\qfrak}\to \triv\to 0$ as representations of the inertia group $I_\qfrak\subseteq \Gal(\Qbar_p/K_\qfrak)$.

In particular if $F'/F$ is an extension of degree $[F':F]\leq 2$, $\pfrak'\subseteq \o_{F'}$ with $\pfrak'\vert \pfrak$ is a prime of $\o_{F'}$ above $\pfrak$, and $\chibar, \chibar': \Gal(\Qbar_p/K_\qfrak)\to (\o_{F'}/\pfrak'^{2n\cdot e(\pfrak'/\pfrak)})^\times$ are characters such that $0\to \chibar\to A[\pfrak^{2n}]\otimes_{\o_\pfrak} \o_{F',\pfrak'}\to \chibar'\to 0$, then $\chibar\equiv \chi_p\pmod*{\pfrak'^{n\cdot e(\pfrak'/\pfrak)}}$ and $\chibar'\equiv \triv\pmod*{\pfrak'^{n\cdot e(\pfrak'/\pfrak)}}$.
\end{lem}

\begin{proof}
The second statement follows from the first by Lemma \ref{recovering characters from a direct sum}.

So let us prove the first. Since $A/K_\qfrak$ has split semistable bad reduction, writing $\bar{A}$ for the special fibre of the connected component of the identity in the N\'{e}ron model of $A$ over $\o_{K,\qfrak}$, there is a canonical exact sequence $0\to \left(\G_m^{\times t}\right)_{/(\o_K/\qfrak)}\to \bar{A}\to B\to 0$ with $B/(\o_K/\qfrak)$ an abelian variety and $t > 0$. But because $\o\inj \End_K(A)$ acts functorially on the N\'{e}ron model, thus the connected component of its identity, thus its special fibre, and thus this decomposition, we find that $\o\inj \End_{\o_K/\qfrak}\left(\left(\G_m^{\times t}\right)_{/(\o_K/\qfrak)}\right) = M_t(\Z)$. We conclude that $t = \dim{A} =: g$, i.e.\ $A/K$ has split totally toric reduction at $\qfrak$.

Now, by Raynaud's uniformization cross, we find that, because $A/K$ has split totally toric reduction at $\qfrak$, there is then a split torus $\left(\G_m^{\times g}\right)_{/\o_{K,\qfrak}}$ uniformizing $A/K_\qfrak$. (This is nothing but the analogue of the theory of the Tate curve for $\GL_2$-type abelian varieties of higher dimension.)

The lemma follows from explicitly computing the $p$-adic and thus $\pfrak$-adic Tate module of the split torus $\left(\G_m^{\times g}\right)_{/\o_{K,\qfrak}}$ by choosing the evident basis (namely one starting with the basis $\left((\zeta_p^{\delta_{i,j}})_{j=1}^g\right)_{i=1}^g$ for $\mu_p^{\times g}$) --- just as in the elliptic curve case.
\end{proof}

Next we explain a corollary of a result of Raynaud which we will in fact not use except in the discussion following its proof.

\begin{lem}\label{past some point everything is truncated barsotti-tate}
Let $p,h,e\in \Z^+$ with $p$ prime. Then: there is an explicit (thus effectively computable) constant $\delta_{p,h,e}\in \Z^+$ depending only on $p$, $h$, and $e$ such that the following holds.
\begin{itemize}
\item Let $K/\Q$ be a number field. Let $\qfrak\subseteq \o_K$ with $\qfrak\vert (p)$ be a prime of $\o_K$ over $p$ with ramification index $e(\qfrak/p)\leq e$. Let $m, n\in \Z^+$ with $m,n\geq \delta_{p,h,e}$. Let $\mathcal{G}/\o_{K,\qfrak}$ be a commutative finite flat group scheme over $\o_{K,\qfrak}$ whose generic fibre $G/K_\qfrak$ corresponds to a Galois module with underlying abelian group $(\Z/p^{m+n})^{\oplus h}$. Write $\mathcal{G}_i/\o_{K,\qfrak}$ for the scheme-theoretic closure of $G[p^i]$ inside $\mathcal{G}$ (thus $\mathcal{G}_i/\o_{K,\qfrak}$ is a commutative finite flat group scheme over $\o_{K,\qfrak}$ as well). Write $\mathcal{G}_{i,j} := \mathcal{G}_j/\mathcal{G}_i$, the quotient taken in the category of commutative finite flat group schemes over $\o_{K,\qfrak}$. Write $G_{i,j} := p^i\cdot G[p^j]$ for the generic fibre of $\mathcal{G}_{i,j}$.

Then: there is a Barsotti-Tate (aka $p$-divisible) group $\widetilde{\mathcal{G}}/\o_{K,\qfrak}$ over $\o_{K,\qfrak}$ such that $\widetilde{\mathcal{G}}_n\iso \mathcal{G}_{m,m+n}$ as commutative finite flat group schemes over $\o_{K,\qfrak}$ --- in other words, $\mathcal{G}_{m,m+n}$ is a truncated Barsotti-Tate group of level $n$ (aka a $\BT_n$) over $\o_{K,\qfrak}$. Moreover, the image of the restriction map $\End_{\o_{K,\qfrak}}(\mathcal{G}_{m,m+n})\to \End_{K_\qfrak}(G_{m,m+n})$ contains the image of the canonical map $\End_{K_\qfrak}(G)\to \End_{K_\qfrak}(G_{m,m+n})$.
\end{itemize}
\end{lem}

\begin{proof}
All but the last sentence follows from Corollary $3.4.5$ of Raynaud's \cite{raynaud}.

As for the statement about endomorphisms, modulo a technical modification our argument is essentially the same one Tate uses to prove the Corollary in Section $4.5$ of his \cite{tate}, except that we use Corollary $3.3.1$ of Raynaud's \cite{raynaud} in place of the uniqueness of prolongations that is available in Tate's situation (and which is also due to Raynaud).

Let $\phi\in \End_{K_\qfrak}(G)$ be a $\Gal(\Qbar_p/K_\qfrak)$-equivariant endomorphism of the Galois module corresponding to $G$. Of course $\phi$ induces a map $G_{m,m+n}\to G_{m,m+n}$, which we will also call $\phi$. Let $\Gamma_\phi := G\times_\phi G\subseteq G\times G$ be the graph of $\phi$. Thus $G\simeq \Gamma_\phi$. Write, as in e.g.\ the Corollary in Section $4.1$ of Tate's \cite{tate}, $\mathcal{G}^-$ and $\mathcal{G}^+$ for the minimal and maximal prolongations of $G$, respectively. Thus there are maps $\mathcal{G}^+\to \mathcal{G}\to \mathcal{G}^-$ over $\o_{K,\qfrak}$ inducing the identity on generic fibres.

Let $\mathcal{H}\subseteq \mathcal{G}\times \mathcal{G}$ be the scheme-theoretic closure of $\Gamma_\phi$ in $\mathcal{G}\times \mathcal{G}$. Then $\mathcal{H}$ is a commutative finite flat group scheme over $\o_{K,\qfrak}$ prolonging $G$ (as is clear from the discussion at the beginning of Section $4.2$ of Tate's \cite{tate}). Therefore it also admits maps $\mathcal{G}^+\to \mathcal{H}\to \mathcal{G}^-$ over $\o_{K,\qfrak}$ inducing the identity on generic fibres. Note also that $\mathcal{H}_{m,m+n}\to \mathcal{G}_{m,m+n}\times \mathcal{G}_{m,m+n}$ as well (as is clear at the level of coordinate rings), where $\mathcal{H}_{m,m+n}/\o_{K,\qfrak}$ is defined in precisely the same way as $\mathcal{G}_{m,m+n}$, except with scheme-theoretic closures taken in $\mathcal{H}$.

These maps in particular induce maps $\mathcal{G}_{m,m+n}^+\to \mathcal{G}_{m,m+n}\to \mathcal{G}_{m,m+n}^-$ and $\mathcal{G}_{m,m+n}^+\to \mathcal{H}_{m,m+n}\to \mathcal{G}_{m,m+n}^-$ over $\o_{K,\qfrak}$ inducing the identity on fibres. But now by Corollary $3.3.1$ of Raynaud's \cite{raynaud} each finite flat group scheme being mapped into or out of is a $\BT_n$ with $n\geq \delta_{p,h,e}$. Therefore, because these maps induce the identity on generic fibres and because $n\geq \delta_{p,h,e}$ is explicitly sufficiently large in terms of $p$, $h$, and $e$, it follows from Corollary $3.4.5$ of Raynaud's \cite{raynaud} that said maps are isomorphisms over $\o_{K,\qfrak}$.

We conclude in particular that there is an isomorphism $\mathcal{G}_{m,m+n}\simeq \mathcal{H}_{m,m+n}$ (namely the inverse of the restriction to $\mathcal{H}_{m,m+n}$ of the projection onto the first factor on $\mathcal{G}_{m,m+n}\times \mathcal{G}_{m,m+n}$) inducing the identity on generic fibres. The desired extension of $\phi$ is given by $\mathcal{G}_{m,m+n}\simeq \mathcal{H}_{m,m+n}\to \mathcal{G}_{m,m+n}\times \mathcal{G}_{m,m+n}\surj \mathcal{G}_{m,m+n}$, where the last map denotes projection onto the second factor.
\end{proof}

Finally let us explain our trick.

Even given Lemma \ref{past some point everything is truncated barsotti-tate} a priori it is not clear how to produce a Barsotti-Tate group \emph{with endomorphisms by $\o_{F',\pfrak'}$} whose Galois representation lifts our given character. At the level of special fibres there is no problem, by  e.g.\ a "crystalline boundedness principle" (see e.g.\ Corollary $1.7$ of Lau-Nicole-Vasiu's \cite{lau-nicole-vasiu} for one such statement, though earlier work of Vasiu also contains results sufficient for the discussion) and Krasner. Deformation from the special fibre is also not an issue \emph{if $F'$ is unramified above $p$} by Proposition $2.6$ of Wedhorn's \cite{wedhorn}. However the general situation (of deforming $\o_{F',\pfrak'}$-structure of a $\BT$ along with a given deformation of its corresponding $\BT_n$ when $F'$ is ramified above $p$) is not so satisfactory --- see e.g.\ the discussion of "the main technical obstacle" on page $230$ of Scholze's \cite{scholze} for a recent example of another work facing such an issue.

Now naturally our situation is extremely special, and because of this one should be able to proceed explicitly via a calculation with Breuil-Kisin modules.

We will not bother, because of the following trick.

\begin{lem}\label{computability of a level past which there are only CM characters}
Let $g\in \Z^+$. Let $K/\Q$ be a number field. Let $p\in \Z^+$ be a prime. Then: there is an effectively computable constant $C_{g,K,p}\in \Z^+$ depending only on $g$, $K$, and $p$ such that the following holds.
\begin{itemize}
\item Let $n\geq C_{g,K,p}$. Let $\qfrak\subseteq \o_K$ with $\qfrak\vert (p)$ be a prime of $\o_K$ above $p$. Let $F/\Q$ be a number field of degree $[F:\Q]\leq g$. Let $\pfrak\subseteq \o_F$ with $\pfrak\vert (p)$ be a prime of $\o_F$ above $p$. Let $\chibar: \Gal(\Qbar_p/K_\qfrak)\to (\o_{F,\pfrak}/p^n)^\times$ be a character such that the Galois module corresponding to $\chibar$ arises as the generic fibre of a finite flat group scheme over $\o_{K,\qfrak}$. Then: there is a subset $S\subseteq \Hom_{\Q\text{-alg.}}(K_\qfrak, \Qbar_p)$ such that $\chibar\vert_{I_\qfrak}\equiv \prod_{\sigma\in S} \sigma\pmod*{p^n}$.
\end{itemize}
\end{lem}

\begin{proof}
There is evidently a finite-time algorithm which determines whether or not an input finite $\Gal(\Qbar_p/K_\qfrak)$-module arises as the generic fibre of a finite flat group scheme over $\o_{K,\qfrak}$ (this amounts to checking whether any of an explicit finite list of orders in an \'{e}tale $K_\qfrak$-algebra are closed under a given comultiplication).

So now let $F'/\Q$ be a number field with a prime $\pfrak'\vert (p)$ such that $F'_{\pfrak'}$ contains all extensions of $\Q_p$ of degree at most $g$. Let $\Phi_n$ be the set of characters $\chibar: \Gal(\Qbar_p/K_\qfrak)\to (\o_{F',\pfrak'}/p^n)^\times$ such that the Galois module corresponding to $\chibar$ arises as the generic fibre of a finite flat group scheme over $\o_{K,\qfrak}$. Let $\Phi_n'\subseteq \Phi_n$ be the subset of characters for which there is no $S\subseteq \Hom_{\Q\text{-alg.}}(K_\qfrak, \Qbar_p)$ such that $\chibar\vert_{I_\qfrak}\equiv \prod_{\sigma\in S} \sigma\pmod*{p^n}$. Of course both $n\mapsto \Phi_n$ and $n\mapsto \Phi_n'$ are effectively computable (via local class field theory and the previous remark).

For $n > m$ reduction modulo $p^m$ induces a map $f_{n\rightarrow m}: \Phi_n\to \Phi_m$ with $f_{n\rightarrow m}(\Phi_n')\subseteq \Phi_m'$.

We claim that, for all $m\in \Z^+$, $\bigcap_{n\geq m} f_{n\rightarrow m}(\Phi_n') = \emptyset$, or in other words (since the $f_{n\rightarrow m}(\Phi_n')$ are decreasing) that $f_{n\rightarrow m}(\Phi_n') = \emptyset$ for $n$ sufficiently large. In other words, we claim that, given a compatible sequence $\chibar_n: \Gal(\Qbar_p/K_\qfrak)\rightarrow (\o_{F',\pfrak'}/p^n)^\times$ of characters whose corresponding Galois modules arise as generic fibres of finite flat group schemes over $\o_{K,\qfrak}$, there is an $S\subseteq \Hom_{\Q\text{-alg.}}(K_\qfrak, \Qbar_p)$ such that $\chibar_n\vert_{I_\qfrak}\equiv \prod_{\sigma\in S} \sigma\pmod*{p^n}$ for all $n$. But such a compatible sequence of characters amounts to a character $\chi: \Gal(\Qbar_p/K_\qfrak)\rightarrow \o_{F',\pfrak'}^\times$ whose corresponding Galois module arises as the generic fibre of a $p$-divisible group over $\o_{K,\qfrak}$, and it is standard that the inertial restrictions of such characters are exactly the characters of "CM $p$-divisible groups", namely the characters $\prod_{\sigma\in S} \sigma$ for $S\subseteq \Hom_{\Q\text{-alg.}}(K_\qfrak, \Qbar_p)$ (for example: $\chi\vert_{I_\qfrak}$ is crystalline, and so, twisting by a suitable product of Lubin-Tate characters (which preserves the conclusion), without loss of generality it has all weights $0$, whence the conclusion follows from Theorem $2$ in Section $3$ of Tate's \cite{tate-p-divisible-groups}).

Finally it remains only to note that the minimal $n\in \Z^+$ for which $f_{n\rightarrow 1}(\Phi_n') = \emptyset$ is effectively computable: starting with $n = 1$, if $f_{n\rightarrow 1}(\Phi_n')\neq \emptyset$ then increment $n\mapsto n+1$ and repeat --- the process will end in finite time exactly because $f_{n\rightarrow 1}(\Phi_n') = \emptyset$ for $n$ sufficiently large.
\end{proof}

\subsubsection{Proof of Lemma \ref{CM lifting}.}

We may now prove Lemma \ref{CM lifting}.

\begin{proof}[Proof of Lemma \ref{CM lifting}.]
Repeat verbatim the proof of Lemma $3.9$ in \cite{my-first-effective-mordell-paper}, and replace the use of Fontaine-Laffaille theory to evaluate the relevant inertial restrictions with Lemma \ref{computability of a level past which there are only CM characters} for good primes and Lemma \ref{description of the inertial representation of a GL2-type abelian variety with bad reduction} for bad primes --- instead of $p\gg_{g,K,S} 1$ one needs only $p^n\gg_{g,K,S} 1$ to force equality in the congruences of global units and then Frobenius traces at the chosen Faltings-Serre set.
\end{proof}

\section{Proof of Theorem \ref{a priori isogeny estimate for GL2-type abelian varieties}.}

\subsection{Reversal of arrows.}

We prove the following standard fact.

\begin{lem}\label{reversal of arrows}
Let $K/\Q$ be a number field. Let $A,B/K$ be abelian varieties over $K$. Let $\phi: A\to B$ be a $K$-isogeny. Then: there is a $K$-isogeny $\phi': B\to A$ such that $\phi\circ \phi' = \deg{\phi}\in \End_K(B)$ and $\phi'\circ \phi = \deg{\phi}\in \End_K(A)$.
\end{lem}

\begin{proof}
Write $n := \deg{\phi}$. Let $G := \ker{\phi}$. Thus $G\subseteq A[n]$. Thus $A[n]/G\subseteq A/G = B$ is a $K$-subgroup of $B/K$. The corresponding quotient map $\phi': B\to B/(A[n]/G)\simeq A$ is defined over $K$ (since the subgroup is stable under $\Gal(\Qbar/K)$) and evidently satisfies $\phi'\circ \phi = n\in \End_K(A)$.

Finally $(\phi\circ \phi' - n)\circ \phi = 0\in \Hom_K(A,B)$ and $\phi$ is surjective, so $\phi\circ \phi' = n\in \End_K(B)$ as well.
\end{proof}

It follows that if there is a $K$-isogeny $A\to B$ of degree $n$, then there is a $K$-isogeny $B\to A$ of degree $n^{2\dim{A}-1}$.

\subsection{Reduction to the special case when $\End_K(A)$ is "maximal".}

Let us define a set $\widetilde{\mathcal{R}}_{g,K,S}'$ of (isomorphism classes of) "maximal" relevant endomorphism rings.

Let $\widetilde{\mathcal{R}}_{g,K,S}^0$ be the explicit finite set of $\Q$-algebras produced by the proof of Corollary \ref{the geometric endomorphism algebra is one of an explicit finite set of possibilities}. For each $D\in \widetilde{\mathcal{R}}_{g,K,S}^0$, we define $\widetilde{\mathcal{R}}_{D,K,S}'$ as follows.

By Lemma \ref{GL2-type abelian varieties are isotypic} it follows that either $D\iso M_{n_1}(K_1)\times M_{n_2}(K_2)$ with $K_i/\Q$ CM and such that $n_1\cdot [K_1:\Q] = n_2\cdot [K_2:\Q] = g$, or else $D$ is simple, and then moreover $D\iso M_n(D')$ with $D'$ either itself a CM field of degree $\frac{g}{n}$ or $\frac{2g}{n}$ over $\Q$, or else a quaternion algebra over a CM field of degree $\frac{g}{2n}$ over $\Q$.

In the former case, i.e.\ when $D\iso M_{n_1}(K_1)\times M_{n_2}(K_2)$, we let $\widetilde{\mathcal{R}}_{D,K,S}' := \{M_{n_1}(\o_{K_1})\times M_{n_2}(\o_{K_2})\}$, a singleton. In the latter case, writing $D\iso M_n(D')$ with $D'$ either a CM field or quaternion algebra, if $D'$ is commutative then its ring of integers $\o_{D'}$ is its unique maximal order, and we let $\widetilde{\mathcal{R}}_{D,K,S}' := \{M_n(\o_{D'})\}$, again a singleton. Otherwise we choose an explicit finite set of representatives $\widetilde{\mathcal{R}}_{D',K,S}'$ of the conjugacy classes of maximal orders of the quaternion algebra $D'$ (note that there are there finitely many such conjugacy classes, by finiteness of the class number (and thus the type number) of a quaternion algebra over a number field), and then let $\widetilde{\mathcal{R}}_{D,K,S}' := \{M_n(\o) : \o\in \widetilde{\mathcal{R}}_{D',K,S}'\}$.

Thus we have defined $\widetilde{\mathcal{R}}_{D,K,S}'$ for each $D\in \widetilde{\mathcal{R}}_{g,K,S}^0$. Let $$\widetilde{\mathcal{R}}_{g,K,S}' := \bigcup_{D\in \widetilde{\mathcal{R}}_{g,K,S}^0} \widetilde{\mathcal{R}}_{D,K,S}'.$$

Now let us show that it suffices to prove Theorem \ref{a priori isogeny estimate for GL2-type abelian varieties} in the special case where $\End_K(A)\in \widetilde{\mathcal{R}}_{g,K,S}'$.

\begin{prop}\label{everybody is isogenous to someone with maximal endomorphisms}
Let $A/K$ be a split semistable $g$-dimensional abelian variety over $K$ with good reduction outside $S$ and such that $\End_K^0(A) = \End_{\Qbar}^0(K)\in \widetilde{\mathcal{R}}_{g,K,S}^0$. Then: there is a $B/K$ with $\End_K(B)\in \widetilde{\mathcal{R}}_{g,K,S}'$ such that $A\sim_K B$.
\end{prop}

\begin{proof}
By construction of $\widetilde{\mathcal{R}}_{g,K,S}'$ it suffices to show this for $A/K$ $K$-simple. When $\End_K(A) = \End_{\Qbar}(K)$ is commutative, this follows from the usual Serre tensor product construction: writing $F := \End_K^0(A)$, automatically $\End_K(A)\subseteq \o_F$ is an order, and $A\sim_K A\otimes_{\End_K(A)} \o_F$ as desired.

The only remaining case is the case of $\End_K^0(A)\iso D\in \mathcal{R}_{g,K,S}^0$ a quaternion algebra. By changing the isomorphism $D\iso \End_K^0(A)$ via an automorphism of $D$ without loss of generality $\End_K(A)\inj \o$ with $\o\in \mathcal{R}_{g,K,S}'$ one of our chosen maximal orders.

Now we again note that $A\sim_K A\otimes_{\End_K(A)} \o$, and we are done, modulo a discussion of the Serre tensor product in this context.

Now the Serre tensor product is standard for modules over a \emph{commutative} ring --- see e.g.\ Proposition $1.7.4.4$ of Chai-Conrad-Oort's \cite{chai-conrad-oort}. However in this situation the relevant ring is noncommutative, and there does not seem to be a treatment in the literature of the construction $A\mapsto A\otimes_{\o'} \o$ for $\o'\subseteq \o$ an order. But it too is easy since we are over a number field: embed into $\C$ and write $A\iso \C^g/\Lambda$ with $\o'\cdot \Lambda = \Lambda$ and $\o\actson \C^g$ via the standard representation ($\o$ is an order in a quaternion algebra over a number field of degree $\frac{g}{2}$ in our situation, whence it naturally acts on $(\C^{\oplus 2})^{\oplus \frac{g}{2}}$). Then the lattice $\o\cdot \Lambda$ contains $\Lambda$ with finite index, whence $A/\C$ is $\C$-isogenous to $A' := \C^g/(\o\cdot \Lambda)$, whence $A'$ is defined over $\Qbar$, whence over a number field $L/K$. But now $A'/L$ represents the functor $R\mapsto A(R)\otimes_{\o'} \o$ over $L$, and so it is uniquely determined up to unique isomorphism. Therefore by Galois descent it is the base change of a variety over $K$, whence we have shown that the functor $R\mapsto A(R)\otimes_{\o'} \o$ over $K$ is representable (by $A'/K$), producing our desired Serre tensor product (which evidently has endomorphisms by $\o$ because e.g.\ it does over $\C$ and $\End_K^0(A') = \End_K^0(A) = \End_{\Qbar}^0(A) = \End_{\Qbar}^0(A')$).
\end{proof}

\subsection{Proof of Theorem \ref{a priori isogeny estimate for GL2-type abelian varieties}.}
Now we may prove Theorem \ref{a priori isogeny estimate for GL2-type abelian varieties}.

\begin{proof}[Proof of Theorem \ref{a priori isogeny estimate for GL2-type abelian varieties}.]
We first reduce to the case when $A/K$ is split semistable and has all its geometric endomorphisms defined over $K$, i.e.\ at every prime of $K$ either $A$ has good reduction or else has split (i.e.\ the relevant torus is split) semistable bad reduction, and also that $\End_K(A) = \End_{\Qbar}(A)$. Let $K'/K$ be the explicit finite Galois extension produced by the proof of Lemma \ref{everything happens over an explicit finite extension}. By Lemma \ref{everything happens over an explicit finite extension}, $A/K'$ is split semistable and has all its geometric endomorphisms defined over $K'$.

Thus in order to show the claimed reduction we need only show that if $\Hom_{K'}(A,B)$ is generated as an abelian group by $K'$-isogenies of degree $\ll_{g,K,S} 1$, then $\Hom_K(A,B)$ is generated as an abelian group by $K$-isogenies of degree $\ll_{g,K,S} 1$. This follows because the (evidently idempotent and thus surjective) $\Z$-linear projection $\pi: \Hom_{K'}^0(A,B)\to \Hom_K^0(A,B)$ via $\phi\mapsto \frac{1}{[K':K]}\sum_{\sigma\in \Gal(K'/K)} \sigma(\phi)$ has the property that $\deg{\left([K':K]\cdot \pi(\phi)\right)}\ll_{g,K,S,\phi} 1$, and thus the image of $[K':K]\cdot \pi: \Hom_{K'}(A,B)\to \Hom_K(A,B)$, which contains $[K':K]\cdot \Hom_K(A,B)$, has index $\ll_{g,K,S} 1$ and is generated as an abelian group by $K$-isogenies of degree $\ll_{g,K,S} 1$. We conclude the reduction by noting that $\Hom_K(A,B)$ is free abelian of rank $\ll_g 1$.

So we conclude that we may simply prove the claim for $A,B/K'$ instead. Thus, replacing $K$ by $K'$, we have achieved the desired reduction: without loss of generality $A/K$ is split semistable and has all its geometric endomorphisms defined over $K$.

Now let us reduce to the case of $A/K$ with "maximal" endomorphism ring. By Proposition \ref{everybody is isogenous to someone with maximal endomorphisms}, there is an $A'/K$ with $A\sim_K A'$ (and thus $A'\sim_K B$) and $\End_K(A')\in \widetilde{\mathcal{R}}_{g,K,S}'$.

By Lemma \ref{reversal of arrows}, without loss of generality we may take $A = A'$ (indeed, if we prove the theorem in this case then we produce $K$-isogenies $A'\to A$ and $A'\to B$ of degree $\ll_{g,K,S} 1$, whence by Lemma \ref{reversal of arrows} a diagram $A\to A'\to B$ with each map a $K$-isogeny of degree $\ll_{g,K,S} 1$).

So we have ensured that $\End_K(A)\in \widetilde{\mathcal{R}}_{g,K,S}'$ as well.

Now let us complete the proof.

In the non-isotypic case we are done\footnote{We could treat this case in a way more akin to the rest of the argument and thus avoid using another isogeny estimate like Theorem \ref{masser-wustholz}, but we will not bother.} --- if $A\simeq A_1^{\times n_1}\times A_2^{\times n_2}$ with each $A_i/K$ $K$-simple and admitting sufficiently many complex multiplications over $K$, then $\End_K(A_i)$ is the maximal order in $\End_K^0(A_i)$ and so $h(A_i)\ll_{g,K} 1$ because we may even construct $A_i/\C$. Therefore $h(A)\ll_{g,K} 1$ and so the theorem follows from e.g.\ Theorem \ref{masser-wustholz}.

Otherwise we have that $A\simeq \widetilde{A}^{\times n}$ is isotypic.

Let us first treat the case of $\End_K^0(A)\simeq M_n(F)\in \widetilde{\mathcal{R}}_{g,K,S}^0$ with $F/\Q$ a CM field of degree $g = [F:\Q]$. Thus by construction $\End_K(\widetilde{A})\simeq \o_F$ and $\End_K(A)\simeq M_n(\o_F)$ under this identification.

So now let $\phi: A\to B$ be a $K$-isogeny, and let $G := \ker{\phi}$ be its kernel. Our task is to show that there is a $\gamma\in \End_K(A)\simeq M_n(\o_F)$ with $\ker{\gamma}\supseteq G$ and $[\ker{\gamma} : G]\ll_{g,K,S} 1$, because then $\gamma: A\to A$ factors as $A\xrightarrow{\phi} B\xrightarrow{\phi'} A$, and $\deg{\phi'}\ll_{g,K,S} 1$.

To produce such a $\gamma\in M_n(\o_F)$ we do the following. Write $G =: \bigoplus_p \widetilde{G}_p$. Let $G_p := M_2(\o_{F,p})\cdot \widetilde{G}_p$ (with $M_2(\o_{F,p})$ acting diagonally on $T_p(A) = T_p(\widetilde{A})^{\oplus n}$). Then by Lemmas \ref{large prime lemma} (for $p\gg_{g,K,S} 1$) and \ref{small prime lemma} (for $p\ll_{g,K,S} 1$) we have that $$\prod_p [G_p : \widetilde{G}_p]\ll_{g,K,S} 1.$$

Let $N_p\in \Z^+$ be such that $G\subseteq A[p^{N_p}]$. Let $\Gamma_p\subseteq T_p(A)$ be the preimage of $G\subseteq A[p^{N_p}]\simeq T_p(A)/p^{N_p}$ in $T_p(A)$. Of course $\Gamma_p\subseteq T_p(A)$ is of finite index, and $M_2(\o_{F,p})\cdot \Gamma_p = \Gamma_p$.

Recall that to define $\rho_{\widetilde{A},p}: \Gal(\Qbar/K)\to \GL_2(\o_{F,p})$ we implicitly chose an isomorphism $T_p(\widetilde{A})\iso \o_{F,p}^{\oplus 2}$ as $\o_{F,p}$-modules (this is possible because $\o_{F,p}$ is a direct sum of principal ideal domains). Thus $T_p(A) = T_p(\widetilde{A})^{\oplus n}\iso (\o_{F,p}^{\oplus 2})^{\oplus n}$ under this choice.

Consequently we may regard $\Gamma_p\subseteq (\o_{F,p}^{\oplus 2})^{\oplus n}$ as an $M_2(\o_{F,p})$-submodule. Let $e_1 := \diag(1,0)\in M_2(\o_{F,p})$. Let $\widetilde{\Gamma}_p := e_1\cdot \Gamma_p$, regarded as an $\o_{F,p}$-submodule of $\o_{F,p}^{\oplus n}\iso \o_{F,p}^{\oplus n}\oplus 0$.

Then because $M_2(\o_{F,p})\cdot \Gamma_p = \Gamma_p$ it follows that $\Gamma_p = \widetilde{\Gamma}_p^{\oplus 2}$ (since $\diag(0,1) = \twobytwo{0}{1}{1}{0}\cdot e_1\cdot \twobytwo{0}{1}{1}{0}$).

Now because $\o_{F,p}$ is a direct sum of principal ideal domains it follows that $\widetilde{\Gamma}_p$ is free (thus of rank $n$) and so there is an $\alpha_p\in M_n(\o_{F,p})$ such that $\widetilde{\Gamma}_p = \alpha_p\cdot \o_{F,p}^{\oplus n}$. Thus $\Gamma_p = \alpha_p\cdot T_p(A)$.

Because $p^{N_p}\cdot T_p(A)\subseteq \Gamma_p$ by definition, it follows that $\beta_p := p^{N_p}\cdot \alpha_p^{-1}$ preserves $T_p(A)$ and thus $\beta_p\in M_n(\o_{F,p})$ satisfies $\beta_p\cdot \alpha_p = p^{N_p}\cdot \id$.

Therefore (since $\Gamma_p = \alpha_p\cdot T_p(A)$) we have found that $\beta_p\cdot G_p = 0$ and indeed $G_p = \ker{(\beta_p\actson A[p^\infty])}$. Of course this same property would hold were we to replace $(\alpha_p, \beta_p)$ by $(\alpha_p\cdot g^{-1}, g\cdot \beta_p)$ with $g\in \GL_n(\o_{F,p})$.

Collecting these cosets $\GL_n(\o_{F,p})\cdot \beta_p$ over all primes $p$, we have produced an element $(\beta_p)_p\in \left(\prod_p \GL_n(\o_{F,p})\right)\backslash \GL_n(\o_F\otimes_\Z \A_{\Q}^{\text{fin.}}) / \GL_n(\o_F\otimes_\Z \Q)$. By Minkowski (this is just the class group of $F$) it follows that there is a $\gamma\in M_n(\o_F)$ such that $\ker{(\gamma\actson A[p^\infty])}\supseteq G_p$ for all $p$ and $$\prod_p [\ker{(\gamma\actson A[p^\infty])} : G_p]\ll_{g,K,S} 1.$$

Hence $\ker{\gamma}\supseteq G$ and $[\ker{\gamma} : G]\ll_{g,K,S} 1$ as desired.

So all that is left to be done is to deal with the case of quaternionic multiplication, i.e.\ the case when $\End_K^0(A)\simeq M_n(D)\in \mathcal{R}_{g,K,S}^0$ with $D/F$ a quaternion algebra over a CM field $F/\Q$ with $[F:\Q] = \frac{g}{2}$. Thus by construction $\End_K^0(\widetilde{A})\simeq D$, $\End_K(\widetilde{A})\simeq \o$, and $\End_K(A)\simeq M_n(\o)$ with $\o\in \widetilde{\mathcal{R}}_{g,K,S}'$ a maximal order of $D$.

When $D$ is split at all primes of $F$ above $p$ the argument producing the coset $\GL_n(\o_p)\cdot \beta_p$ is exactly the same as above since $D_p\simeq M_2(F_p)$ and $\o_p\simeq M_2(\o_{F,p})$ under that identification. In particular this is the case for $p\gg_{g,K,S} 1$.

Otherwise $p\ll_{g,K,S} 1$ and we proceed as follows. Write $F_p\simeq \bigoplus_{\pfrak\vert (p)} F_\pfrak$, $D_p\simeq \bigoplus_{\pfrak\vert (p)} D_\pfrak$, and $\o_p\simeq \bigoplus_{\pfrak\vert (p)} \o_\pfrak$.

Note that again $\o_p$ is a direct sum of principal ideal domains, though now the summands are noncommutative.\footnote{If $D/F$ is ramified at $\pfrak$, then the unique maximal order $\o_\pfrak\subseteq D_\pfrak$ is the set of those elements of $D_\pfrak$ which have integral norm (as can easily be checked in the explicit basis arising from $D_\pfrak\simeq (\pi, u)_\pfrak$ as we used in Proposition \ref{the endomorphism algebra is one of an explicit finite set of possibilities}). Therefore it follows that the fractional ideals of $\o_\pfrak$ are simply $i^k\cdot \o_\pfrak$ with $i^2 = \pi$ the uniformizer of $\o_{F,\pfrak}$ and $k\in \Z$. In particular $\o_\pfrak$ is a noncommutative principal ideal domain.}

Again let $\phi: A\to B$ be a $K$-isogeny and let $G := \ker{\phi}$. Let $G =: \bigoplus_p \widetilde{G}_p$ and  $G_p := \o_p^\opp\cdot \widetilde{G}_p$ where the left action of $\o_p^\opp$ arises via $\o_p^\opp\subseteq D_p^\opp$ and the fact that $D_p^\opp\inj M_{2gn}(\Q_p)$ (acting by right multiplication on the $2gn$-dimensional $\Q_p$-vector space $D_p^{\oplus n}$) is the commutant of $\End_K^0(A)\simeq M_n(D_p)\inj M_{2gn}(\Q_p)$ (acting by left multiplication on said vector space).

By for example repeating the argument that gives Smith normal form ("Jacobson normal form") one finds that, for all ramified primes $\pfrak$ of $D/F$, all torsion-free finitely-generated modules over $\o_\pfrak$ are free (choose a generating set over $\Z_p$ and then a generating set over $\Z_p$ of the submodule of relations and apply Jacobson normal form to diagonalize the resulting matrix). In particular it follows that $T_\pfrak(\widetilde{A})\iso \o_\pfrak^{\opp}$ as $\o_\pfrak^{\opp}$-modules when $D/F$ is ramified at $\pfrak$, since $T_\pfrak(\widetilde{A})$ is a free $\o_\pfrak^\opp$-module by the above. (We already used this when writing $\rho_{\widetilde{A}, p}: \Gal(\Qbar/K)\to (\o_p^\opp)^\times$.)

Therefore we may simply repeat the same argument. Let $N_p\in \Z^+$ be such that $G\subseteq A[p^{N_p}]\simeq T_p(A)/p^{N_p}$ and let $\Gamma_p\subseteq T_p(A)\simeq T_p(\widetilde{A})^{\oplus n}$ be the preimage of $G$ in $T_p(A)$ as before. Thus $\Gamma_p$ is a torsion-free finitely-generated $\o_p^\opp$-module, whence free (automatically of rank $n$), whence there is again an $\alpha_p\in M_n(\o_p)$ such that $\Gamma_p = \alpha_p\cdot T_p(A)$. We again produce $\beta_p\in M_n(\o_p)$ such that $\beta_p\cdot \alpha_p = p^{N_p}\cdot \id$ and then by Minkowski (this is a very explicit case of Borel-Harish-Chandra) it follows after considering the class $(\beta_p)_p\in \left(\prod_p \GL_n(\o\otimes_\Z \Z_p)\right)\backslash \GL_n(\o\otimes_\Z \A_{\Q}^{\text{fin.}}) / \GL_n(\o\otimes_\Z \Q)$ that there is a $\gamma\in M_n(\o)$ such that $\ker{\gamma}\supseteq G$ and $[\ker{\gamma} : G]\ll_{g,K,S} 1$, exactly as before.

We are done.
\end{proof}

\section{Proof of Theorem \ref{the endomorphism ring is one of an explicit finite set of possibilities}.}

\begin{proof}[Proof of Theorem \ref{the endomorphism ring is one of an explicit finite set of possibilities}.]
Apply Theorem \ref{a priori isogeny estimate for GL2-type abelian varieties} with $B = A$.\footnote{Alternatively, combine Theorem \ref{a priori isogeny estimate for GL2-type abelian varieties}, Lemma \ref{everything happens over an explicit finite extension}, and Propositions \ref{the endomorphism algebra is one of an explicit finite set of possibilities} and \ref{everybody is isogenous to someone with maximal endomorphisms} to upper bound the index of $\End_K(A)$ in one of an explicit finite set of rings.}
\end{proof}

\section{Proof of Theorem \ref{upper bound on the number of GL2-type abelian varieties}.}

\begin{proof}[Proof of Theorem \ref{upper bound on the number of GL2-type abelian varieties}.]
The number of $K$-isogeny classes is $\ll_{g,K,S} 1$ by Lemma \ref{faltings' lemma} and Faltings' proof of the Tate conjecture for homomorphisms of abelian varieties. Each $K$-isogeny class contains $\ll_{g,K,S} 1$ abelian varieties by Theorem \ref{a priori isogeny estimate for GL2-type abelian varieties}. The theorem follows.
\end{proof}

\section{Proof of Corollary \ref{upper bound on the number of S-integral K-points on a Hilbert modular variety}.}

\begin{proof}[Proof of Corollary \ref{upper bound on the number of S-integral K-points on a Hilbert modular variety}.]
It suffices (by Chevalley-Weil) to bound the number of $\o_{K,S}$-points on the Hilbert modular variety with full level-$10^{10}!$ structure (so that we are actually bounding the number of points on a variety). By Theorem \ref{upper bound on the number of GL2-type abelian varieties} the only things to bound are the number of embeddings $\o\inj \o'$ with $\o'\in \mathcal{R}_{g,K,S}$ (notation as in Theorem \ref{the endomorphism ring is one of an explicit finite set of possibilities}) up to conjugation by $\o'^\times$, and, given $\o\inj \End_K(A)$, the number of $\o$-linear polarizations $A\otimes_{\o} \afrak\simeq A^*$ with $\afrak\in \Cl(\o)$. The latter is easy. The former is straightforward and follows from e.g.\ an easily effectivized case of Borel--Harish-Chandra \cite{borel-harish-chandra} (use e.g.\ the usual generalization of Eichler's trace formula for optimal embeddings).
\end{proof}

\section{Proof of Theorem \ref{serre open image theorem for GL2-type abelian varieties}.\label{proof of serre open image theorem section}}

Let us now prove Theorem \ref{huge image after some point}.\footnote{Because we chose to state Theorem \ref{huge image after some point} in a previous section (namely immediately after Lemma \ref{irreducibility after some point} in Section \ref{irreducibility after some point section}), it is important to note that no preceding results depend on Theorem \ref{huge image after some point}, and so we are e.g.\ free to use Proposition \ref{large prime lemma} in the below proof.}

\begin{proof}[Proof of Theorem \ref{huge image after some point}.]
Let $L/K$ be the compositum of all quadratic extensions of $K$ which are unramified outside $S$ --- $L/K$ is an explicit finite extension by Minkowski's proof of the Hermite-Minkowski theorem. We take $C_{g,K,N}$ to be the constant produced by Proposition \ref{large prime lemma} applied to $(g,L,S)$, where $S$ is the finite set of all places of $L$ which divide $N$. Let then $p\geq C_{g,K,N}$.

By Lemma $3.9$ of our \cite{my-first-effective-mordell-paper} aka Lemma \ref{irreducibility after some point} above, there is a $g\in \GL_2(\o_E/\pfrak)$ and a subfield $\F_q\subseteq \o_E/\pfrak$ such that $$g\cdot \SL_2(\F_q)\cdot g^{-1}\subseteq \overline{\rho}_{A,\pfrak}(\Gal(\Qbar/K))\subseteq (\o_E/\pfrak)^\times\cdot (g\cdot \GL_2(\F_q)\cdot g^{-1}),$$ so that all there is to show is that $\F_q = \o_E/\pfrak$. Of course by changing basis (i.e.\ replacing $\rho_{A,\pfrak}$ by $g^{-1}\cdot \rho_{A,\pfrak}\cdot g$) without loss of generality $g = \id$.

Let $\sigma\in \Gal((\o_E/\pfrak)/\F_q)\subseteq \Gal((\o_E/\pfrak)/\F_p)\simeq \Gal(E_\pfrak/\Q_p)$. Regarding $\sigma\in \Gal(E_\pfrak/\Q_p)$, let $\rho_{A,\pfrak}^\sigma := \sigma\circ \rho_{A,\pfrak}$ be the conjugate representation.

Now for each $g\in \Gal(\Qbar/K)$ there is a $\lambda_g\in (\o_E/\pfrak)^\times$ and an $f(g)\in \GL_2(\F_q)$ such that $\rhobar_{A,\pfrak}(g) = \lambda_g\cdot f(g)\pmod*{\pfrak}$. (Note that $g\mapsto \lambda_g$ is well-defined modulo $\F_q^\times$.) Because $\lambda_g^2\cdot \det{f(g)} = \det{\rhobar_{A,\pfrak}(g)} = \chi_p(g)\in \F_p^\times$, it follows that, for all $g\in \Gal(\Qbar/K)$, $\lambda_g^2\in \F_q^\times$.

Therefore in particular $\sigma(\lambda_g) = \pm \lambda_g$ for all $g\in \Gal(\Qbar/K)$. It follows that, for each $g\in \Gal(\Qbar/K)$, there is a unique $\eps(g)\in \{\pm 1\}$ such that $\sigma(\rhobar_{A,\pfrak}(g))\cdot \rhobar_{A,\pfrak}(g)^{-1}\ = \eps(g)\cdot \id\pmod*{\pfrak}$. Since $\lambda_{gh}\equiv \lambda_g\cdot \lambda_h\pmod*{\F_q^\times}$, it follows that $\eps: \Gal(\Qbar/K)\to \{\pm 1\}$ is a character.

We claim that $\eps$ is unramified outside $S$. This amounts to the statement that $\eps\vert_{I_\qfrak} = \triv$ for $\qfrak\subseteq \o_K$ with $\qfrak\vert (p)$ a prime of $K$ above $p$ (it is automatic for all other residue characteristics). Of course (since wild inertia is pro-$p$) $\eps$ is at most tamely ramified at $\qfrak$. It then follows that if $\eps\vert_{I_\qfrak}\neq \triv$ then $\eps\vert_{I_\qfrak}\equiv \chi_p\vert_{I_\qfrak}^{\frac{p-1}{2}}\pmod*{\pfrak}$. So we need only rule this latter possibility out.

Because $\rhobar_{A,\pfrak}^\sigma\simeq \rhobar_{A,\pfrak}\otimes \eps$ and both $\rhobar_{A,\pfrak}\vert_{\Gal(\Qbar_p/K_\qfrak)}$ and $\rhobar_{A,\pfrak}^\sigma\vert_{\Gal(\Qbar_p/K_\qfrak)}$ correspond to Galois modules which prolong to finite flat group schemes over $\o_{K,\qfrak}$, it follows from Corollary $3.4.4$ of Raynaud's \cite{raynaud-pp} (or else the results of Fontaine-Laffaille we used in the proof of Lemma $3.9$ of \cite{my-first-effective-mordell-paper}) that $\eps\vert_{I_\qfrak}\not\equiv \chi_p\vert_{I_\qfrak}^{\frac{p-1}{2}}\pmod*{\pfrak}$, since fundamental characters can only occur with multiplicity at most $1$. Indeed, writing e.g.\ $\alpha: I_\qfrak\to \Fbar_p^\times$ for a character occurring in $\rhobar_{A,\pfrak}\vert_{I_\qfrak}\otimes_{\o_E/\pfrak} \Fbar_p$ (which semisimplifies as an $I_\qfrak$-representation to a sum of characters because $I_\qfrak^{\mathrm{tame}}$ is abelian and pro-prime-to-$p$), by Corollary $3.4.4$ of Raynaud's \cite{raynaud-pp} $\alpha$ is a multiplicity-free product of fundamental characters of $I_\qfrak^{\mathrm{tame}}$ --- and, applying the same reasoning to $(\rhobar_{A,\pfrak}\otimes \eps)\vert_{I_\qfrak}\otimes_{\o_E/\pfrak} \Fbar_p$, so is $\alpha\cdot \eps\vert_{I_\qfrak}$, whence $\eps\vert_{I_\qfrak}$ cannot be $\chi_p\vert_{I_\qfrak}^{\frac{p-1}{2}}\pmod*{\pfrak}$.

So $\eps$ is unramified outside $S$. Therefore $\eps\vert_{\Gal(\Qbar/L)} = \triv$ by definition of $L/K$.

So we have found that, for all $\sigma\in \Gal((\o_E/\pfrak)/\F_q)$, $\sigma(\rhobar_{A,\pfrak}(g)) = \rhobar_{A,\pfrak}(g)$ for all $g\in \Gal(\Qbar/L)$. In other words $\rhobar_{A,\pfrak}(\Gal(\Qbar/L))\subseteq \GL_2(\F_q)$, and so $\F_p[\rhobar_{A,\pfrak}(\Gal(\Qbar/L))]\subseteq M_2(\F_q)$.

However by Proposition \ref{large prime lemma} we have that $\Z_p[\rho_{A,\pfrak}(\Gal(\Qbar/L))] = M_2(\o_{E,\pfrak})$, whence $\F_p[\rhobar_{A,\pfrak}(\Gal(\Qbar/L))] = M_2(\o_E/\pfrak)$, whence $\F_q = \o_E/\pfrak$ as desired.
\end{proof}

Having proven Theorem \ref{huge image after some point}, we may now prove Theorem \ref{serre open image theorem for GL2-type abelian varieties}.

\begin{proof}[Proof of Theorem \ref{serre open image theorem for GL2-type abelian varieties}.]
Let $\text{\^{\i}}_{g,K,S}$ be the constant produced by Theorem \ref{huge image after some point} on input $(g,K,\prod_{\pfrak\in S} \Nm\,{\pfrak})$.\footnote{Note that, by e.g.\ Theorem \ref{the endomorphism ring is one of an explicit finite set of possibilities}, without loss of generality all $p\geq \text{\^{\i}}_{g,K,S}$ are prime to the discriminant of $\o$, whence e.g.\ $\o$ is maximal at all primes above such $p$.} Let $\widetilde{H} := \left(\prod_{p\geq \text{\^{\i}}_{g,K,S}} \rho_{A,p}\right)(\Gal(\Qbar/K))\subseteq \prod_{p\geq \text{\^{\i}}_{g,K,S}} G_\o(\Z_p) =: \widetilde{G}$. Of course $\widetilde{H}$ is compact and thus $\widetilde{H}\subseteq \widetilde{G}$ is closed ($\widetilde{G}$ is Hausdorff), and $\det{\widetilde{H}} = \det{\widetilde{G}} = \prod_{p\geq \text{\^{\i}}_{g,K,S}} \Z_p^\times$ since each $p\geq \text{\^{\i}}_{g,K,S}$ is unramified in $K$ and so $K$ and $\Q(\{\zeta_n : (n, \text{\^{\i}}_{g,K,S}!) = 1\})$ are linearly disjoint. So it suffices to show that $H = G$, where $H := \widetilde{H}\cap \ker{(\det)}$ and similarly $G := \widetilde{G}\cap \ker{(\det)} = \prod_{p\geq \text{\^{\i}}_{g,K,S}}\prod_{\pfrak\vert (p)} \SL_2(\o_\pfrak)$ (whence automatically $H\subseteq G$).

Let $H_\pfrak := \rho_{A,\pfrak}(\Gal(\Qbar/K))\cap \ker{(\det)}\subseteq \SL_2(\o_\pfrak) =: G_\pfrak$. Note that Theorem \ref{huge image after some point} amounts to the statement that $H_\pfrak\pmod*{\pfrak} = G_\pfrak\pmod*{\pfrak}$ (since $\SL_2(\o/\pfrak)$ is its own commutator subgroup).

We claim that then $H_\pfrak = G_\pfrak$.

Indeed if $g\in G_\pfrak$, then let $h_0\in H_\pfrak$ be such that $h\equiv g\pmod*{\pfrak}$. On replacing $g$ by $h_0^{-1}\cdot g$, without loss of generality $g\equiv \id\pmod*{\pfrak}$. Now we proceed as in Serre's proof of Lemma $3$ on page $\text{IV-}23$ of his \cite{serre-abelian-l-adic-representations-and-elliptic-curves}.

Because $p\geq \text{\^{\i}}_{g,K,S}$ is so large, it is unramified in $\o$. Write $g - \id =: p\cdot (u + v + w)$ with $u$ a multiple of $\twobytwo{1}{-1}{1}{-1}$ modulo $\pfrak$, $v$ strictly lower triangular modulo $\pfrak$, and $w$ strictly upper triangular modulo $\pfrak$ (this is possible because $\tr(g - \id)\equiv 0\pmod*{\pfrak^2}$ since $\det{g} = 1$). Let $h_u, h_v, h_w\in H_\pfrak = \SL_2(\o/\pfrak)$ be such that $h_u\equiv \id + u\pmod*{\pfrak}$, $h_v\equiv \id + v\pmod*{\pfrak}$, and $h_w\equiv \id + w\pmod*{\pfrak}$. Let $h_1 := h_u^p\cdot h_v^p\cdot h_w^p$. Thus $h_1\equiv \id + p\cdot (u + v + w)\equiv g\pmod*{\pfrak^2}$. Thus $H_\pfrak\pmod*{\pfrak^2} = G_\pfrak\pmod*{\pfrak^2}$ and so on replacing $g$ by $h_1^{-1}\cdot g$ without loss of generality $g\equiv \id\pmod*{\pfrak^2}$.

Now if $H_\pfrak\pmod*{\pfrak^n} = G_\pfrak\pmod*{\pfrak^n} = \SL_2(\o/\pfrak^n)$ and $g\equiv \id\pmod*{\pfrak^n}$ with $n\geq 2$, then write $g\equiv \id + p^n\cdot z\pmod*{\pfrak^{n+1}}$ (thus $\tr(z)\equiv 0\pmod*{\pfrak}$) and let $h_z\in H_\pfrak$ be such that $h_z\equiv \id + p^{n-1}\cdot z\pmod*{\pfrak^n}$. Letting $h_n := h_z^p$ we see that $h_n\equiv \id + p^n\cdot z\pmod*{\pfrak^{n+1}}$, and so $H_\pfrak\pmod*{\pfrak^{n+1}} = G_\pfrak\pmod*{\pfrak^{n+1}}$ and we may replace $g$ by $h_n^{-1}\cdot g$ to continue the induction.

Therefore we see that $H_\pfrak\pmod*{\pfrak^n} = G_\pfrak\pmod*{\pfrak^n}$ for all $n\in \N$, and so because $H_\pfrak\subseteq G_\pfrak$ is closed it follows that $g\in H_\pfrak = G_\pfrak$, whence the claim.

Next let $H_p := \rho_{A,p}(\Gal(\Qbar/K))\cap \ker{(\det)}\subseteq \SL_2(\o_p) =: G_p$. We claim that $H_p = G_p$.

Let us first see that it suffices to show that $H_p\pmod*{p} = G_p\pmod*{p}$. Indeed we may simply repeat the above aka "take the product over all $\pfrak\vert (p)$ of the above argument": if $H_p\pmod*{p} = G_p\pmod*{p} = \SL_2(\o/p)$, then, given a $g\in G_p\pmod*{p^2}$ with $g\equiv \id\pmod*{p}$, write $g =: \id + p\cdot (u + v + w)$ with $u\in (\o/p)\cdot \twobytwo{1}{-1}{1}{-1}$, $v\in (\o/p)\cdot \twobytwo{0}{0}{1}{0}$, $w\in (\o/p)\cdot \twobytwo{0}{1}{0}{0}$, which is again possible because $\det{g} = 1\in \o_p^\times$ and so $\tr(g - \id)\equiv 0\pmod*{p^2}$. Let $h_u$, $h_v$, $h_w\in H_p$ be such that $h_u\equiv \id + u\pmod*{p}$ and similarly for $h_v$ and $h_w$. Let $h_0 := h_u^p\cdot h_v^p\cdot h_w^p\in H_p$ and note that $h_0\equiv \id + p\cdot (u + v + w)\pmod*{p^2}$ and so $H_p\pmod*{p^2} = G_p\pmod*{p^2}$. The induction step also follows verbatim.

So it suffices to show that $H_p\pmod*{p} = G_p\pmod*{p}$.

To see this write $\{\pfrak_1, \ldots, \pfrak_m\} := \{\pfrak\subseteq \o : \pfrak\vert (p)\}$ and let $H_p^{(i)} := \left(\prod_{j=1}^i \rho_{A,\pfrak_j}\pmod*{\pfrak_j}\right)(\Gal(\Qbar/K))\cap \ker{(\det)}\subseteq \prod_{j=1}^i \SL_2(\o/\pfrak_i) =: G_p^{(i)}$. Thus in particular we have already seen that $H_p^{(1)} = G_p^{(1)}$, and our claim, namely that $H_p\pmod*{p} = G_p\pmod*{p}$, amounts to the statement that $H_p^{(m)} = G_p^{(m)}$ (remember that $p$ is unramified).

Now if $H_p^{(i)} = G_p^{(i)}$, then since $H_p^{(i+1)}\subseteq G_p^{(i+1)} = G_p^{(i)}\times \SL_2(\o/\pfrak_{i+1})$ and $H_p^{(i+1)}$ surjects onto both factors (since $H_{\pfrak_{i+1}}\pmod*{\pfrak_{i+1}} = G_{\pfrak_{i+1}}\pmod*{\pfrak_{i+1}}$), by Goursat's Lemma it follows that $H_p^{(i+1)}$ is the preimage of the graph of an isomorphism $G_p^{(i)}/N\iso \SL_2(\o/\pfrak_{i+1})/N'$ with $N\subseteq G_p^{(i)}$ and $N'\subseteq \SL_2(\o/\pfrak_{i+1})$ normal.

If $N' = \SL_2(\o/\pfrak_{i+1})$ then $N = G_p^{(i)}$ and so $H_p^{(i+1)} = G_p^{(i+1)}$ as desired, so that we may continue the induction. Otherwise because $\PSL_2(\o/\pfrak_{i+1})$ is simple it follows by considering the normal subgroup $(\{\pm \id\}\cdot N')/\{\pm \id\}\subseteq \PSL_2(\o/\pfrak_{i+1})$ that either $N'\subseteq \{\pm \id\}$, or else $\{\pm \id\}\cdot N' = \SL_2(\o/\pfrak_{i+1})$. The latter implies that $N' = \SL_2(\o/\pfrak_{i+1})$ because e.g.\ $\SL_2(\o/\pfrak_{i+1})$ is its own commutator subgroup, a contradiction. So we deduce that $N'\subseteq \{\pm \id\}$.

Thus $G_{\pfrak_{i+1}}/N'$ is either $\PSL_2(\o/\pfrak_{i+1})$ or $\SL_2(\o/\pfrak_{i+1})$. Now consider the surjection $G_p^{(i)} = \prod_{j=1}^i \SL_2(\o/\pfrak_j)\surj G_p^{(i)}/N\iso \SL_2(\o/\pfrak_{i+1})/N'$. As we have seen the only nonzero quotients of $\SL_2(\o/\pfrak_j)$ are either $\PSL_2(\o/\pfrak_j)$ or $\SL_2(\o/\pfrak_j)$.

The map $\prod_{j=1}^i \SL_2(\o/\pfrak_j)\surj \SL_2(\o/\pfrak_{i+1})/N'$ amounts to a product of maps $\SL_2(\o/\pfrak_j)\to \SL_2(\o/\pfrak_{i+1})/N'$ whose images commute and generate $\SL_2(\o/\pfrak_{i+1})/N'$. In particular at least one is nonzero. Let $k$ be minimal such that the map $\SL_2(\o/\pfrak_k)\to \SL_2(\o/\pfrak_{i+1})/N'$ is nonzero. It follows that e.g.\ the kernel of the map $\SL_2(\o/\pfrak_k)\to \SL_2(\o/\pfrak_{i+1})/N'$ is either trivial or else $\{\pm \id\}$.

Of course the induced map $\prod_{j=1}^i \SL_2(\o/\pfrak_j)\surj \PSL_2(\o/\pfrak_{i+1})$ is also surjective. Consider the induced map $\SL_2(\o/\pfrak_k)\to \SL_2(\o/\pfrak_{i+1})/N'\surj \PSL_2(\o/\pfrak_{i+1})$ --- it is also nonzero since $p\geq \text{\^{\i}}_{g,K,S}$ is large. Therefore it has image isomorphic to either $\SL_2(\o/\pfrak_k)$ or else $\PSL_2(\o/\pfrak_k)$, and in particular $p$ divides the order of its image. By Dickson's classification of subgroups of $\PSL_2(\o/\pfrak_{i+1})$, it follows that said image is either conjugate into a Borel, or else conjugate to a subgroup containing $\PSL_2(\F_p)$.

It cannot be conjugate into a Borel because otherwise $\PSL_2(\o/\pfrak_k)$ would be solvable, a contradiction. So it has a conjugate containing $\PSL_2(\F_p)$. Conjugating the map $\prod_{j=1}^i \SL_2(\o/\pfrak_j)\surj \PSL_2(\o/\pfrak_{i+1})$ shows that, for all $j\neq k$, the image of $\SL_2(\o/\pfrak_j)$ commutes with the standard $\PSL_2(\F_p)\subseteq \PSL_2(\o/\pfrak_{i+1})$ (since it commutes with the image of $\SL_2(\o/\pfrak_k)$), whence it is trivial. Consequently for $j\neq k$ the image of $\SL_2(\o/\pfrak_j)\to \SL_2(\o/\pfrak_{i+1})$ is trivial, since it must lie inside $\{\pm \id\}$ but $\SL_2(\o/\pfrak_j)$ is its own commutator subgroup.

We conclude that the surjective map $\prod_{j=1}^i \SL_2(\o/\pfrak_j)\surj \SL_2(\o/\pfrak_{i+1})/N'$ factors as projection onto the $k$-th factor composed with $\SL_2(\o/\pfrak_k)\surj \SL_2(\o/\pfrak_{i+1})/N'$. In particular $N = \left(\prod_{j=1}^{k-1} \SL_2(\o/\pfrak_j)\right)\times N''\times \left(\prod_{j=k+1}^{i+1} \SL_2(\o/\pfrak_j)\right)$ with $N''\subseteq \{\pm \id\}\subsetneq \SL_2(\o/\pfrak_k)$ normal.

Again consider $\SL_2(\o/\pfrak_k)\surj \PSL_2(\o/\pfrak_{i+1})$. Because the map is surjective and $\PSL_2(\o/\pfrak_{i+1})$ has trivial centre, it follows that the map factors through a surjection $\PSL_2(\o/\pfrak_k)\surj \PSL_2(\o/\pfrak_{i+1})$. Because both groups are simple it follows that this map must be an isomorphism. In particular $\o/\pfrak_k\iso \o/\pfrak_{i+1}$ as fields. Because the automorphisms of $\PSL_2(\o/\pfrak_k)$ are given by a composition of a field automorphism with conjugation (by an element of $\GL_2(\o/\pfrak_k)$), it follows that, up to precomposition with conjugation by an element of $\GL_2(\o/\pfrak_k)$, the map $\SL_2(\o/\pfrak_k)\surj \SL_2(\o/\pfrak_{i+1})/N'$ induces the map $\PSL_2(\o/\pfrak_k)\to \PSL_2(\o/\pfrak_{i+1})$ induced by a field isomorphism $\o/\pfrak_k\iso \o/\pfrak_{i+1}$.

In other words the map $\SL_2(\o/\pfrak_k)\surj \PSL_2(\o/\pfrak_{i+1})$ factors as $\SL_2(\o/\pfrak_k)\iso \SL_2(\o/\pfrak_{i+1})\surj \PSL_2(\o/\pfrak_{i+1})$, where the latter map induces the identity $\PSL_2(\o/\pfrak_{i+1})\to \PSL_2(\o/\pfrak_{i+1})$. Of course this then forces the latter map to be the canonical quotient. Note that we also have that a map $\SL_2(\F_q)\surj \SL_2(\F_q)$ which induces the identity $\PSL_2(\F_q)\simeq \PSL_2(\F_q)$ must also be the identity (for example by writing each element of $\SL_2(\F_q)$ as a product of commutators of elements of $\SL_2(\F_q)$).

So we finally conclude the following about our subgroup $H_p^{(i+1)}\subseteq \prod_{j=1}^{i+1} \SL_2(\o/\pfrak_j)$: there is a $g\in \GL_2(\o/\pfrak_{i+1})$, a $k\in \{1, \ldots, i\}$, and a field isomorphism $\phi: \o/\pfrak_k\simeq \o/\pfrak_{i+1}$ such that $$H_p^{(i+1)} = \left\{(m_j)_{j=1}^{i+1}\in \prod_{j=1}^{i+1} \SL_2(\o/\pfrak_j) : g\cdot m_{i+1}\cdot g^{-1} = \phi(m_k)\right\}.$$

But by Proposition \ref{large prime lemma} we know that $\F_p[H_p^{(i+1)}] = \prod_{j=1}^{i+1} M_2(\o/\pfrak_j)$, and the above subgroup evidently generates a proper $\F_p$-subalgebra of $\prod_{j=1}^{i+1} M_2(\o/\pfrak_j)$, a contradiction.

So we conclude that after all $N' = \SL_2(\o/\pfrak_{i+1})$, or in other words that $H_p^{(i+1)} = G_p^{(i+1)}$ and so we may continue the induction. Thus by induction $H_p^{(m)} = G_p^{(m)}$ as desired.

Consequently $H_p = G_p$. 

So now let us conclude by showing $H = G$. Since $H\subseteq G$ is closed, it suffices to show that, for all $\ell\geq \text{\^{\i}}_{g,K,S}$, $G_\ell\subseteq H$, where $G_\ell\iso G_\ell\times \prod_{\ell\neq p\geq \text{\^{\i}}_{g,K,S}} \{\id\}$ is regarded as a normal subgroup of $\prod_{p\geq \text{\^{\i}}_{g,K,S}} G_p$ in the evident way. Similarly regarding $H\cap G_\ell$ as a subgroup of $G_\ell$ (and a normal subgroup of $H$), evidently $H\cap G_\ell\subseteq H_\ell = G_\ell$, and to show that $G_\ell\subseteq H$ it suffices to show that $H\cap G_\ell = G_\ell$. As we have seen (in e.g.\ reducing proving $H_p = G_p$ to showing only that $H_p\pmod*{p} = G_p\pmod*{p}$), to show that $H\cap G_\ell = G_\ell$, it suffices to show that the canonical map $H\cap G_\ell\inj G_\ell\surj \prod_{\lambda\vert (\ell)} \SL_2(\o/\lambda)$ is surjective.

But because $H/(H\cap G_\ell)\inj \prod_{\ell\neq p\geq \text{\^{\i}}_{g,K,S}} G_p$, and because the right-hand side has no quotients of the form $\PSL_2(\F_q)$ with $q$ a power of $\ell$ (since the kernel of the canonical map $\SL_2(\o_\pfrak)\surj \SL_2(\o/\pfrak)$ is solvable via the filtration by principal congruence subgroups modulo the powers of $\pfrak$), it follows that indeed $H\surj H_\ell = G_\ell\surj \prod_{\lambda\vert (\ell)} \SL_2(\o/\lambda)$. Indeed the normal subgroup, say $N\subseteq H$, generated by $H\cap G_\ell$ and the kernel of $H\surj \prod_{\lambda\vert (\ell)} \SL_2(\o/\lambda)$ has the property that $H/N$ is a finite group which surjects upon no nontrivial simple groups (since $\prod_{\lambda\vert (\ell)} \SL_2(\o/\lambda)\surj H/N$ and the former only surjects upon $\PSL_2(\F_q)$ with $q$ a power of $\ell$, while also $\prod_{\ell\neq p\geq \text{\^{\i}}_{g,K,S}} G_p\surj H/N$, and we have seen that the former does not surject upon $\PSL_2(\F_q)$ with $q$ a power of $\ell$), whence it must be trivial.

So $N = H$, which is to say that the map $H\cap G_\ell\inj G_\ell\surj \prod_{\lambda\vert (\ell)} \SL_2(\o/\lambda)$ is surjective, as desired.

So $H\cap G_\ell = G_\ell$, whence $G_\ell\subseteq H$, whence for all $g\in \prod_{p\geq \text{\^{\i}}_{g,K,S}} G_p$ and for all $N\in \Z^+$ there is an $h_N\in H$ such that, for all $\text{\^{\i}}_{g,K,S}\leq p\leq N$, $h_N^{-1}\cdot g$ has trivial $p$-th coordinate, whence because $H\subseteq G$ is closed it follows that $g\in H$, so that $H = G$ as desired and we are done.
\end{proof}

\section{Proof of Corollary \ref{serre open image theorem for fake elliptic curves}.}

\begin{proof}[Proof of Corollary \ref{serre open image theorem for fake elliptic curves}.]
Let $X_{\o}(10^{10}!)/\Q$ be the corresponding Shimura curve with full $10^{10}!$-level structure. By Chevalley-Weil there is an explicit $L/K$ and a $\widetilde{P}\in X_{\o}(10^{10}!)(L)$ such that $\widetilde{P}\mapsto P$ under $X_{\o}(10^{10}!)\to X_{\o}$. The abelian variety corresponding to $\widetilde{P}$ is simply $A/L$, and the claim for $A/K$ follows from the claim for $A/L$ since $\Gal(\Qbar/L)\subseteq \Gal(\Qbar/K)$. But because all $10^{10}!$-torsion is defined over $L$ it follows that $A/L$ is semistable and so has good reduction everywhere (since $B\inj \End_L^0(A)$ does not act on $\G_m^{\times 2}$). We conclude by applying Lemma \ref{everything happens over an explicit finite extension} and then Theorem \ref{serre open image theorem for GL2-type abelian varieties}.
\end{proof}

\section{Proof of Corollary \ref{un peu d'effectivite for sirin curves}.}

\begin{proof}[Proof of Corollary \ref{un peu d'effectivite for sirin curves}.]
Given a \c{s}irin $C/K$ it is evident (via e.g.\ brute force) that there is an effectively computable \c{s}irin family $(C, K, L, \widetilde{C}, \phi, \pi)$. By Chevalley-Weil there is an explicit $\widetilde{L}/L$ such that all $P\in C(K)$ lift to a $\widetilde{P}\in \widetilde{C}(\widetilde{L})$. By spreading out the family $\pi: A\to \widetilde{C}$ there is an explicit finite set $S$ of places of $\widetilde{L}$ such that, for each such $\widetilde{P}$, the corresponding abelian variety $A_{\widetilde{P}}/\widetilde{L}$ has good reduction outside $S$ (and is by hypothesis of $\GL_2$-type over $\widetilde{L}$). The corollary now follows from Theorem \ref{upper bound on the number of GL2-type abelian varieties}.\footnote{One also has to bound the number of times a given $B/\widetilde{L}$ can occur as a fibre of the nonisotrivial family $\pi: A\to \widetilde{C}$ --- such a bound follows via Zarhin's trick applied to the generic fibre of $(A\times A^*)^{\times 4}\to \widetilde{C}$, whence said multiplicity is explicitly bounded in terms of the degree of the corresponding compactified map $\widetilde{C}\to \overline{A}_g$ with $g := 8\cdot \dim_{\widetilde{C}}{A}$ (since by Theorem \ref{the endomorphism ring is one of an explicit finite set of possibilities} the endomorphism ring of $(A_{\widetilde{P}}\times A_{\widetilde{P}}^*)^{\times 4}\sim_{\widetilde{L}} A_{\widetilde{P}}^{\times 8}$ is one of an explicit finite set of possibilities, it is a simple matter to estimate the number of direct factors of a relevant $(B\times B^*)^{\times 4}/\widetilde{L}$).}
\end{proof}

\section{Sketch of proof of Corollary \ref{an explicit form of a theorem of dimitrov}.}

\begin{proof}[Sketch of proof of Theorem \ref{an explicit form of a theorem of dimitrov}.]
Over $F$, Lemma $3.9$ of our \cite{my-first-effective-mordell-paper} amounts to an effectivization of Propositions $3.1$ and $3.5$ in Dimitrov's \cite{dimitrov}. An effective form of his Proposition $3.8$ then follows from his arguments using his Theorem $3.7$ (which is standard and due to Dickson). This gives the first two parts of the theorem.

To prove the third we must also effectivize his Proposition $3.17$ (from which an effective form of its Corollary $3.18$ follows). However the technique is precisely the same as in our proof of Lemma $3.9$ of our \cite{my-first-effective-mordell-paper}: where Dimitrov uses (using his notation in the proof of his Proposition $3.17$ and slightly paraphrasing) the argument "$c(f_\tau, v)\equiv \eps(v) c(f,v)\pmod*{\mathcal{P}}$ for infinitely many $\mathcal{P}$ implies $c(f_\tau, v) = \eps(v) c(f,v)$", we substitute our use of Faltings-Serre sets --- i.e.\ require such a congruence at only an explicit finite set of primes $v$ of the Galois closure of $F$ --- to obtain an explicit estimate.
\end{proof}

\appendix
\section{Finiteness of isogeny classes of (fake) elliptic curves.\label{the appendix}}

We are not sure if the following is interesting, but in case it is let us quickly explain why the arguments we used to prove Theorem \ref{a priori isogeny estimate for GL2-type abelian varieties} give a simple proof of the following.

\begin{thm}\label{finiteness of isogeny classes of fake elliptic curves}
Let $K/\Q$ be a number field. Let $A/K$ be either an elliptic curve or a fake elliptic curve.\footnote{$A/K$ is a "fake elliptic curve" if and only if it is an abelian surface $A/K$ admitting $D\inj \End_K^0(A)$ with $D/\Q$ a nonsplit quaternion algebra over $\Q$.} Then: there are only finitely many $B/K$ with $A\sim_K B$.
\end{thm}\noindent
The small primes analysis of Proposition \ref{small prime lemma} may even e.g.\ be replaced by a compactness argument.

In any case, given Lemma \ref{faltings' lemma}, in order to repeat the proof of Lemma \ref{irreducibility after some point} the only thing to show is the following, which we do by repeating an argument of Serre.

\begin{thm}\label{tate conjecture for endomorphisms of fake elliptic curves}
Let $K/\Q$ be a number field. Let $A/K$ be either a non-potentially-CM elliptic curve or a non-potentially-CM fake elliptic curve. Then: there is a $C_{K,A}\in \Z^+$ such that for all $p\geq C_{K,A}$ we have that $\Q_p[\rho_{A,p}(\Gal(\Qbar/K))] = M_2(\Q_p)$.
\end{thm}\noindent
(In case $A/K$ is a fake elliptic curve we are implicitly taking $C_{K,A}$ so large that $p\geq C_{K,A}$ implies that $\End_K^0(A)\otimes_{\Q} \Q_p\simeq M_2(\Q_p)$, whence we may write $\rho_{A,p}: \Gal(\Qbar/K)\to \GL_2(\Z_p)$.)

Since the argument is essentially Serre's we will be quite brief.

\begin{proof}[Sketch of proof of Theorem \ref{tate conjecture for endomorphisms of fake elliptic curves}.]
Suppose otherwise. Then $\Q_p[\rho_{A,p}(\Gal(\Qbar/K))]\subseteq M_2(\Q_p)$ is abelian. Thus $(\rho_{A,p}\otimes_{\Z_p} \Qbar_p)^\mathrm{s.s.}\simeq \alpha\oplus \beta$ with $\alpha, \beta: \Gal(\Qbar/K)\to \Qbar_p^\times$ characters. Said characters automatically arise from $p$-divisible groups since they are subquotients of $\rho_{A,p}\otimes_{\Z_p} \Qbar_p$, so they are $p$-adic realizations of algebraic Hecke characters. Thus by Brauer-Nesbitt we see that $(\rho_{A,\ell}\otimes_{\Z_\ell} \Qbar_\ell)^{\mathrm{s.s.}}\simeq \alpha\oplus \beta$ (where now $\alpha$ and $\beta$ denote the corresponding $\ell$-adic realizations) for all large $\ell$, so without loss of generality by changing $p$ we may assume that $\alpha, \beta: \Gal(\Qbar/K)\to \Z_p^\times$ and $0\to \alpha\to \rho_{A,p}\to \beta\to 0$. Now considering the inertial restrictions of $\alpha$ and $\beta$ at primes above $p$ we see that either $\beta\vert_{I_\qfrak} = \triv$ for all $\qfrak\vert (p)$ and thus $|\beta(\mfrak)| = 1$ for all large primes $\mfrak\subseteq \o_K$, contradicting Weil, or else for some $\qfrak\vert (p)$ the sequence $0\to \alpha\vert_{I_\qfrak}\to \rho_{A,p}\vert_{I_\qfrak}\to \beta\vert_{I_\qfrak}\to 0$ splits (via the connected-\'{e}tale sequence), whence by Serre-Tate $A/K_\qfrak$ is CM, another contradiction.
\end{proof}

\renewcommand{\refname}{References.}

\bibliographystyle{plain}

\bibliography{unpeudeffectivitepourlesvarietesmodulairesdehilbertblumenthal}

\end{document}